\documentclass[a4paper,10pt,reqno]{amsart}

\usepackage[english]{babel}
\usepackage[utf8]{inputenc}
\usepackage{amsmath,amssymb,amsfonts,amsthm,amscd}
\usepackage[mathscr]{eucal}
\usepackage{latexsym}

\usepackage{enumitem}
\setlist[itemize]{topsep=0pt, leftmargin=2em}
\setlist[enumerate]{topsep=0pt, leftmargin=2em}

\usepackage{xcolor}
\usepackage{hyperref}
\usepackage{graphicx}

\DeclareMathOperator{\dist}{dist}

\DeclareMathOperator{\arccosh}{arccosh}

\newtheorem{theorem}{Theorem}[section]
\newtheorem*{theorem*}{Theorem}
\newtheorem{proposition}[theorem]{Proposition}
\newtheorem{corollary}[theorem]{Corollary}
\newtheorem{lemma}[theorem]{Lemma}

\theoremstyle{remark}
\newtheorem{remark}[theorem]{Remark}

\numberwithin{equation}{section}
\title[A construction of CMC surfaces in $\mathbb{H}^2\times\mathbb{R}$ and the Krust property]{A construction of constant mean curvature surfaces in $\mathbb{H}^2\times\mathbb{R}$ and the Krust property}
\date{}

\author{Jesús Castro-Infantes}
\address{Departamento de Geometría y Topología, Universidad de Granada, Spain}
\email{jcastroinfantes@ugr.es}

\author{José M. Manzano}
\address{Departamento de Matemáticas, Universidad de Jaén, Spain}
\email{jmprego@ujaen.es}

\author{Magdalena Rodríguez}
\address{Departamento de Geometría y Topología and IMAG (Institute of Mathematics of Granada), Universidad de Granada, Spain}
\email{magdarp@ugr.es}



\begin{document}

\begin{abstract}
We show the existence of a $2$-parameter family of properly Alexandrov-embedded surfaces with constant mean curvature $0\leq H\leq\frac{1}{2}$ in ${\mathbb{H}^2\times\mathbb{R}}$. They are symmetric with respect to a horizontal slice and a $k$ vertical planes disposed symmetrically, and extend the so called minimal saddle towers and $k$-noids. We show that the orientation plays a fundamental role when $H>0$ by analyzing their conjugate minimal surfaces in $\widetilde{\mathrm{SL}}_2(\mathbb{R})$ or $\mathrm{Nil}_3$. We also discover new complete examples that we call $(H,k)$-nodoids, whose $k$ ends are asymptotic to vertical cylinders over curves of geodesic curvature $2H$ from the convex side, often giving rise to non-embedded examples if $H>0$.

In the discussion of embeddedness of the constructed examples, we prove that the Krust property does not hold for any $H>0$, i.e., there are minimal graphs over convex domains in $\widetilde{\mathrm{SL}}_2(\mathbb{R})$, $\mathrm{Nil}_3$ or the Berger spheres, whose conjugate surfaces with constant mean curvature $H$ in $\mathbb{H}^2\times\mathbb{R}$ are not graphs.
\end{abstract}

\maketitle

\section{Introduction}

The theory of constant mean curvature surfaces ($H$-surfaces for short) in homogeneous 3-manifolds has drawn considerable attention during the last decades, being the homogeneous setting a natural step in the generalization of classical results in space forms. This can be exemplified by the half-space theorems given by Daniel and Hauswirth~\cite{DH} and by Hauswirth, Rosenberg and Spruck~\cite{HRS}, or by the recent classification of $H$-spheres in metric Lie groups obtained by Meeks, Mira, Pérez and Ros~\cite{MMPR}. Among simply connected homogeneous 3-manifolds, those whose isometry group has dimension $4$ form a $2$-parameter family $\mathbb{E}(\kappa,\tau)$ with $\kappa-4\tau^2\neq 0$. They contain the product spaces $\mathbb{H}^2\times\mathbb{R}=\mathbb{E}(-1,0)$ and $\mathbb{S}^2\times\mathbb{R}=\mathbb{E}(1,0)$, as well as the Lie groups $\mathrm{SU}(2)$, $\mathrm{Nil}_3$ and $\widetilde{\mathrm{SL}}_2(\mathbb{R})$ endowed with some special left-invariant Riemannian metrics, see~\cite{Dan,Man,MMPR}.

Daniel~\cite{Dan} revealed the existence of a Lawson-type isometric correspondence for $H$-surfaces in $\mathbb{E}(\kappa,\tau)$-spaces, which has become an effective tool in the construction of $H$-surfaces in $\mathbb{H}^2\times\mathbb{R}$ by means of conjugating a certain minimal
surface in $\mathbb{E}(4H^2-1,H)$, e.g., see~\cite{HST,P,MR,MaR,R,ManTor,Ple15,Ple13,CM}. This discloses the
fact that $H$-surfaces in $\mathbb{H}^2\times\mathbb{R}$ with $H=0$, $0<H<\frac{1}{2}$, $H=\frac{1}{2}$ and $H>\frac{1}{2}$ display different features, for their conjugate surfaces live in geometrically different ambient spaces. For instance, compact $H$-surfaces (resp.\ complete $H$-graphs) in $\mathbb{H}^2\times\mathbb{R}$ exist if and only if $H>\frac{1}{2}$ (resp.\ $0\leq H\leq\frac{1}{2}$). The value $H=\frac 12$, known in the literature as   critical mean curvature of $\mathbb{H}^2\times\mathbb{R}$, stands out for the fact that all complete $\frac{1}{2}$-graphs are entire~\cite{HRS}, whilst $H=0$ is the only value such that there are entire graphs contained in a horizontal slab~\cite{HMR2}. Note that $H$-cylinders (i.e., everywhere vertical surfaces that project onto a curve of $\mathbb{H}^2$ of constant geodesic curvature $2H$) also behave differently if they project onto circles ($H>\frac{1}{2}$), horocycles ($H=\frac{1}{2}$) or curves equidistant to a geodesic ($0\leq H<\frac{1}{2}$).

The aforesaid conjugate technique has been specially relevant in the case $H=0$, partly because Hauswirth, Sa Earp and Toubiana~\cite{HST} extended a theorem of Krust (unpublished, see~\cite[Theorem 2.4.1]{K}) showing that the conjugate minimal surface of a minimal graph in a product space $M\times\mathbb{R}$, where $M$ is a Riemannian surface with non-positive constant Gauss curvature, defined on a convex domain of $M$ must be a graph. This result has turned out to be very useful to deal with the embeddedness of the conjugate minimal surface (see for instance~\cite{K, MRT, MR, CM}). Although the Krust property was expected to extend to the cases $0<H\leq\frac{1}{2}$, in the present paper we show it does not hold for any $H>0$.

\begin{theorem}\label{th1}
  For any $H>0$, there are minimal graphs over convex domains in $\mathbb{E}(4H^2-1,H)$ whose conjugate $H$-surfaces in $\mathbb{H}^2\times\mathbb{R}$ are not embedded.
\end{theorem}

It would be very interesting to find some extra hypothesis to prove a Krust-type theorem for $0<H\leq\frac{1}{2}$.  The second author and Torralbo have recently investigated more specific tools to analyze the embeddedness of the conjugate surfaces, see~\cite{MT20}.

For $0< H\leq \frac 12$, the counterexamples in Theorem~\ref{th1} belong to a broader $2$-parameter family of highly symmetric $H$-surfaces in $\mathbb{H}^2\times\mathbb{R}$ that also depend on an integer $k\geq 2$ indicating the number of vertical planes of symmetry. We describe them as follows (more precise depictions will be given in Section~\ref{sec:constructions}):

\begin{theorem}\label{th3}
For each $H\in[0,\frac{1}{2}]$ and for each integer $k\geq 2$, there is a continuous $2$-parameter family of proper Alexandrov-embedded $H$-surfaces $\overline\Sigma_{a,b}^*\subset\mathbb{H}^2\times\mathbb{R}$, $a,b\in(0,\infty]$ not both of them equal to $\infty$, invariant by mirror symmetry about a horizontal slice and about a equiangular family of $k$ vertical planes containing a common vertical geodesic.
\begin{enumerate}[label=\emph{(\alph*)}]
  \item If $a,b<\infty$, then $\overline\Sigma_{a,b}^*$ is called a \emph{saddle tower}. It is singly periodic in the vertical direction, having genus $0$ and $2k$ ends in the quotient of $\mathbb{H}^2\times\mathbb{R}$ by a vertical translation, each end asymptotic to the quotient of half an $H$-cylinder.
  
  \item If $a=\infty$ and $b<\infty$, then $\overline\Sigma_{\infty,b}^*$ is called a \emph{$(H,k)$-noid} (or $H$-catenoid if $k=2$) and has genus $0$ and $k$ ends. If $H\in[0,\frac{1}{2})$, then each end is embedded and asymptotic to (and contained in the concave part of) a vertical $H$-cylinder. If $H=\frac{1}{2}$, then each end is tangent to the ideal boundary $\partial_\infty\mathbb{H}^2\times\mathbb{R}$ along a vertical ideal geodesic.
  
  \item If $a<\infty$ and $b=\infty$, then $\overline\Sigma_{a,\infty}^*$ is called a \emph{$(H,k)$-nodoid} (or $H$-catenodoid if $k=2$) and has genus $0$ and $k$ ends. If $H\in[0,\frac{1}{2})$, then each end is embedded and asymptotic to (and contained in the convex part of) a vertical $H$-cylinder. If $H=\frac 12$, then the corresponding $\frac12$-cylinder disappears at infinity.
\end{enumerate}
\end{theorem}

To describe the asymptotic behavior of these surfaces
  at infinity, we have considered the compactification of
  $\mathbb{H}^2\times\mathbb{R}$ obtained as the product of the
  compactifications of each factor, which is homeomorphic to
  $\overline D\times[-1,1]$, being $\overline D\subset\mathbb{R}^2$
  the closed unit disk. Consequently, we identify
  $\mathbb{S}^1=\partial\overline D$ with
  $\partial_\infty\mathbb{H}^2$, the asymptotic boundary $\partial_\infty (\mathbb{H}^2\times\mathbb{R})$ of $\mathbb{H}^2\times\mathbb{R}$ with $(\mathbb{S}^1\times[-1,1])\cup(\overline D\times\{\pm1\})$,  and $\mathbb{H}^2\times\{\pm\infty\}$ corresponds to $D\times\{\pm1\}$. We will say that a point $p \in \partial _\infty (\mathbb{H}^2\times\mathbb{R})$ belongs to the asymptotic boundary $\partial_\infty\Sigma$ of a surface $\Sigma\subset\mathbb{H}^2\times\mathbb{R}$ if there is a diverging sequence of points $p_n\in\Sigma$ such that $p_n$ converges to $p$ in the compactification.

\begin{figure}
\begin{center}
\includegraphics[width=6cm]{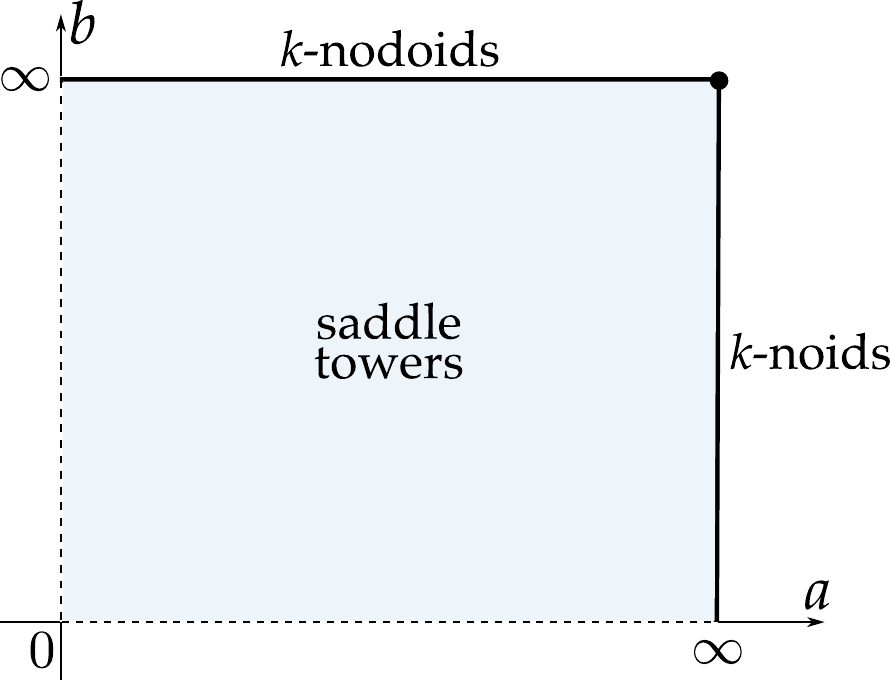}
\caption{Moduli space of the surfaces $\overline\Sigma_{a,b}^*$ for $H\in[0,\frac{1}{2}]$ and $k\geq 2$ fixed. The black dot represents the limit $\overline\Sigma_{\infty,\infty}^*$ (see Section~\ref{sec:constructions}) and the dashed lines indicate the minimal $k$-noids in $\mathbb{R}^3$ obtained in the limit after rescaling the metric.}\label{fig:moduli}
\end{center}
\end{figure}

The surfaces given by Theorem~\ref{th3}, include the minimal saddle towers and the minimal $k$-noids given by Morabito and the third author in~\cite{MR} (in the symmetric case) and by Pyo~\cite{P,P2} independently; the constant mean curvature $k$-noids constructed by Plehnert~\cite{Ple15}; and the catenoids with constant mean curvature
$\frac{1}{2}$ found by Daniel and Hauswirth in~\cite{DH}. This is why we call the examples $\overline\Sigma_{a,b}^*$ saddle towers and the examples $\overline\Sigma_{\infty,b}^*$ $(H,k)$-noids (or $H$-catenoids if $k=2$). These names are in turn inspired by their classical counterparts in $\mathbb{R}^3$. The examples $\overline\Sigma_{a,\infty}^*$ share topology and symmetry with the $(H,k)$-noids, but some of them are not embedded and their horizontal curves of symmetry have self-intersections (similar to a fundamental piece of the generating curve of the classical nodoids). This fact has inspired the name of $(H,k)$-nodoid (or $H$-catenodoid when $k=2$) for these examples. We would also like to point out the analogy with the case $H>\frac{1}{2}$, where horizontal Delaunay surfaces have been constructed by similar methods in~\cite{MT20}.

Our main contribution is to extend Plehnert's construction of $(H,k)$-noids by considering more general Jenkins-Serrin problems and to analyze the embeddedness of the constructed surfaces by means of a careful study of the orientation. The mean curvature $H$ of $\overline\Sigma_{a,b}^*$ becomes the bundle curvature of the ambient space of the conjugate minimal surface $\overline\Sigma_{a,b}\subset\mathbb{E}(4H^2-1,H)$. Thus when $H=0$ the surfaces $\overline\Sigma_{a,b}^*$ and $\overline\Sigma_{b,a}^*$ are congruent for all $a,b\in(0,\infty]$, so the subfamilies of $(0,k)$-noids and $(0,k)$-nodoids are congruent (see Remark~\ref{rmk:minimal-case}). When $H\neq 0$, different directions of rotation of the normal along a vertical geodesic $\Gamma\subset\overline\Sigma_{a,b}$ imply different curvature estimates of the conjugate curve $\Gamma^*\subset\overline\Sigma_{a,b}^*$, see Equation~\eqref{eqn:theta} and Lemma~\ref{lemma:rotation-curvature}. In particular, we exploit the fact that swapping the $+\infty$ and $-\infty$ boundary values in a Jenkins-Serrin problem (i.e., swapping $a$ and $b$ in our construction) outputs very different $H$-surfaces. The differences can be also illustrated by the case of spherical helicoids in Berger spheres, in which different directions of rotation of such an helicoid lead to rotationally invariant unduloid-type or nodoid-type surfaces in $\mathbb{H}^2\times\mathbb{R}$ with $H>\frac{1}{2}$, as observed by Torralbo and the second author~\cite[Proposition~1]{ManTor}. Likewise, vertical helicoids in $\widetilde{\mathrm{SL}}_2(\mathbb{R})$ or $\mathrm{Nil}_3$ give rise to embedded and non-embedded rotationally invariant $H$-surfaces in $\mathbb{H}^2\times\mathbb{R}$ with $0<H\leq\frac{1}{2}$, which are described, e.g., in~\cite[Proposition~5.2]{NSST} and the references therein.

If $H\in[0,\frac{1}{2})$, we emphasize that the ends of the $(H,k)$-noids lie in the concave side of their asymptotic vertical $H$-cylinders, whereas the ends of the $(H,k)$-nodoids lie in their convex side. We will show that the $(H,k)$-noid $\overline\Sigma_{\infty,b}^*$ (resp.\ the $(H,k)$-nodoid $\overline\Sigma_{a,\infty}^*$) is embedded if $b$ (resp.\ $a$) is large enough, since the limit $\overline\Sigma_{\infty,\infty}^*$ is equal to an ideal Scherk $H$-graph. The limit of $\overline\Sigma_{a,b}^*$ when $a\to 0$ or $b\to 0$ is the union of $k$ asymptotic vertical $H$-cylinders that produce, after a suitable rescaling, a symmetric minimal $k$-noid in $\mathbb{R}^3$, see Remark~\ref{rmk:rescaling}. As a matter of fact, for small values of the parameters, $H$-catenoids and $H$-catenodoids desingularize two tangent vertical $H$-cylinders from the concave or the convex side,   respectively. In Section~\ref{sec:krust}, we will show that there is a unique $H$-catenoid or $H$-catenodoid for each prescribed signed distance (in the admissible range) between the asymptotic vertical planes. We will also find that $H$-catenoids are embedded but there are embedded and non-embedded $H$-catenodoids (see Figure~\ref{Fig-Catenoids}). As for the case $H=\frac{1}{2}$, the possible asymptotic horocylinders of $(\frac{1}{2},k)$-noids or $(\frac{1}{2},k)$-nodoids turn out to disappear at infinity, so the $\frac{1}{2}$-catenoids are properly embedded bigraphs, whereas $\frac{1}{2}$-catenodoids are not embedded. Either way, we will give explicitly the domain over which they project (see Figure~\ref{fig:critical-catenoid}). As a matter of fact, catenoids   and catenodoids have some similarities to some  rotationally invariant surfaces with mean curvature~$1$ in the hyperbolic space $\mathbb{H}^3$, see~\cite[Section~6]{UY}.

The first two authors~\cite{CM} have also generalized the symmetric examples in~\cite{MR} to produce minimal saddle towers and $k$-noids with genus one, $k\geq 3$. An interesting open problem is to explore the existence of positive genus non-minimal saddle towers, $(H,k)$-noids and $(H,k)$-nodoids. Notice that $H$-surfaces with finite positive genus in $\mathbb{H}^2\times\mathbb{R}$ with $0<H<\frac{1}{2}$ have not been constructed yet.

The paper is organized as follows. In section~\ref{sec:preliminaries}, we will discuss the role of vertical geodesics in conjugate constructions and prove that the conjugation is continuous with respect to the standard convergence of $H$-surfaces (see Proposition~\ref{prop:continuity}). This will help us understand the asymptotic behavior of the conjugate surfaces of minimal graphs over unbounded domains produced as limits of some Jenkins-Serrin problems over compact domains. Section~\ref{sec:JS} will be devoted to obtain the minimal surfaces in $\widetilde{\mathrm{SL}}_2(\mathbb{R})$ or $\mathrm{Nil}_3$ that will play the role of initial surfaces in our construction by conjugation. In the case of $\widetilde{\mathrm{SL}}_2(\mathbb{R})$, we need to extend some results by Younes~\cite{Y} and Melo~\cite{Me}. In the case of $\mathrm{Nil}_3$ we will obtain horizontal minimal helicoids (see Lemma~\ref{lemma:helicoids-nil}) that are foliated by straight lines and act as barriers to solve the Jenkins-Serrin problem over a strip (the uniqueness of such solutions is an open question). They extend the horizontal helicoids obtained by Daniel and Hauswirth~\cite{DH} (see also~\cite{Ple13}), since we can cover all speeds for both directions of rotation along the horizontal axis, which highlights again the role of orientation in our discussion. Lemma~\ref{lemma:helicoids-nil} also gives examples of entire minimal graphs in $\mathrm{Nil}_3$ foliated by straight lines. Finally, in Section~\ref{sec:constructions} we will undertake the construction of the conjugate surfaces $\overline\Sigma_{a,b}^*$ and complete the proof of Theorem~\ref{th3}. We will also discover that there are non-embedded saddle towers (in contrast to the minimal case) as a consequence of a thorough study of their angle function (see Remark~\ref{rmk:height} and Proposition~\ref{prop:slab}). We will also observe that embeddedness may fail either because the fundamental piece is not embedded (as in the case of some $(H,k)$-nodoids, see Proposition~\ref{prop:knodoids-embeddedness}) or because it does not lie in the appropriate slab or half-space and produces self-intersections after extending it across the boundary (as in the case of some saddle towers, see Proposition~\ref{prop:slab}). Finally, in Section~\ref{sec:krust} we will deal with the proof of Theorem~\ref{th1}.

We would like to remark that $(H,k)$-noids and $(H,k)$-nodoids are parabolic for all $H\in[0,\frac{1}{2})$ and all $k\geq 2$ (they are conformally the Riemann sphere minus $k$ points). Indeed, using~\cite{HMR1} we can prove that the Jenkins-Serrin minimal graphs of $\widetilde{\mathrm{SL}}_2(\mathbb{R})$ used in our construction have finite total curvature. Hence both $(H,k)$-noids and $(H,k)$-nodoids also have finite total curvature when $H<\frac 1 2$.  Daniel and Hauswirth~\cite{DH} proved that $\frac{1}{2}$-catenoids are parabolic. We expect the $(\frac{1}{2},k)$-noids and $(\frac{1}{2},k)$-nodoids to have finite total curvature (in particular they would be parabolic). Properly immersed minimal surfaces with  finite total curvature are characterized as those with finite topology whose ends are asymptotic to geodesic polygons at infinity~\cite{HMR1}. It is expected that a similar result holds for $H\in(0,\frac{1}{2}]$.

\medskip

\noindent\textbf{Acknowledgement.} This research is supported by
MCIN/AEI/10.13039/501100011033/ grants no. PID2020-117868GB-I00 (first
and third authors) and PID2019-111531GA-I00 (second author). The first
and third authors are also supported by FEDER/Andalucía grants no.\
A-FQM-139-UGR18 and P18-FR-4049 and by MINECO/FEDER grant no.\
MTM2017-89677-P. The second author is also supported by the Ramón y
Cajal grant no.\ RYC2019-027658-I.

\section{Preliminaries}\label{sec:preliminaries}

Given $\kappa,\tau\in\mathbb{R}$, the space $\mathbb{E}(\kappa,\tau)$ is the unique simply connected oriented 3-manifold admitting a Riemannian submersion with constant bundle curvature $\tau$ over $\mathbb{M}^2(\kappa)$, the simply-connected surface with constant curvature $\kappa$, whose fibers are the integral curves of a unit Killing vector field $\xi$. The bundle curvature can be characterized by the equation $\overline\nabla_X\xi=\tau X\times\xi$, where $\overline\nabla$ stands for the Levi-Civita connection and its sign depends on the orientation of $\mathbb{E}(\kappa,\tau)$ by means of the cross product $\times$, see~\cite{Dan,Man} and the references therein. In our constructions, we will employ the so-called cylinder model
\begin{equation}\label{eqn:model-Ekt}
\mathbb{E}(\kappa,\tau)=\left\{(x,y,z)\in\mathbb{R}^3:1+\tfrac{\kappa}{4}(x^2+y^2)>0\right\},
\end{equation}
endowed with the Riemannian metric
\[\mathrm{d} s^2=\lambda^2(\mathrm{d} x^2+\mathrm{d} y^2)+(\mathrm{d} z+\lambda\tau(y\mathrm{d} x-x\mathrm{d} y))^2,\]
where $\lambda=\left(1+\tfrac{\kappa}{4}(x^2+y^2)\right)^{-1}$. We also prescribe the orientation such that
\begin{align*}
 E_1&=\frac{\partial_x}{\lambda}-\tau y\,\partial_z,& 
 E_2&=\frac{\partial_y}{\lambda}+\tau x\,\partial_z& \mbox{and }& &
 E_3&=\partial_z
\end{align*}
define a global positively oriented orthonormal frame. The Killing submersion reads $\pi(x,y,z)=(x,y)$ over $\mathbb{M}^2(\kappa)$ identified with the disk of radius $\tfrac{2}{\sqrt{-\kappa}}$ and center the origin in the $(x,y)$-plane if $\kappa<0$ or with the whole $(x,y)$-plane if $\kappa\geq0$. If $\kappa>0$, then this model is locally isometric to a Berger sphere, but it is not complete. Note that the Riemannian metric $\lambda^2(\mathrm{d} x^2+\mathrm{d} y^2)$ has constant curvature $\kappa$ and the unit Killing vector field is given by $\xi=E_3$.

The bundle curvature can be interpreted geometrically in terms of holonomy of horizontal lifts of curves as follows. Assume that $\alpha$ is a piecewise-$\mathcal{C}^1$ Jordan curve in $\mathbb{M}^2(\kappa)$ that encloses a relatively compact region $U$ with the interior of $U$ to the left when traveling on $\alpha$. If we consider a horizontal lift $\overline\alpha$ in $\mathbb{E}(\kappa,\tau)$ such that $\pi\circ\overline\alpha=\alpha$, then the endpoints of $\overline\alpha$ can be joined by a vertical geodesic of length $2\tau\mathrm{Area}(U)$, see~\cite[Proposition~3.3]{Man}. This makes a big difference between the cases $\tau=0$ and $\tau\neq 0$ that we will exploit in Lemma~\ref{lem:slab}.

\subsection{Conjugation of $H$-surfaces}

Let $\Sigma$ be a simply connected oriented Riemannian surface and let $H\geq 0$. Daniel~\cite{Dan} and Hauswirth, Sa Earp and Toubiana~\cite{HST} discovered a Lawson-type isometric correspondence between minimal immersions $\phi:\Sigma\to\mathbb{E}(4H^2-1,H)$ and $H$-immersions $\phi^*:\Sigma\to\mathbb{H}^2\times\mathbb{R}=\mathbb{E}(-1,0)$. Consider the fundamental data $(\nu,T,A)$ of $\phi$ consisting of the angle function $\nu=\langle N,\xi\rangle$, the tangent component $T=\xi-\nu N$ of the unit Killing vector field $\xi$ and the shape operator $A$ associated to the unit normal $N$ used in the computation of the mean curvature. Likewise, let $(\nu^*,T^*,A^*)$ be the
fundamental data of $\phi^*$ associated to the unit normal $N^*$. The correspondence will be called \emph{conjugation} in the sequel and is characterized by the following three relations: 
\begin{align}
\nu^*&=\nu,&
T^*&=JT,&
A^*&=JA+H\,\mathrm{id},\label{eqn:daniel}
\end{align}
where $J$ stands for a $\frac\pi2$-rotation in the tangent bundle of $\Sigma$. More precisely, the orientation of $\Sigma$ and the orientations of the ambient spaces fulfill the compatibility relations $\mathrm{d}\phi_p(Jv)=N\times\mathrm{d}\phi_p(v)=N^*\times\mathrm{d}\phi^*_p(v)$ for all tangent vectors $v\in T_p\Sigma$. The immersions $\phi$ and $\phi^*$ determine each other up to isometries preserving the orientation of the ambient spaces and the orientation of vertical fibers. Throughout the text we will write $\Sigma$ and $\Sigma^*$ to refer to conjugate (immersed) surfaces, though it is more convenient to keep on using the notation of immersions $\phi$ and $\phi^*$ to prove the results in the present section.

Let $\Gamma\subset\Sigma$ be a regular curve. Then $\phi(\Gamma)$ is part of a horizontal (resp.\ vertical) geodesic of $\mathbb{E}(4H^2-1,H)$ if and only if the conjugate curve $\phi^*(\Gamma)$ lies in a vertical plane (resp.\ horizontal slice) $P\subset\mathbb{H}^2\times\mathbb{R}$ intersecting $\phi^*(\Sigma)$ orthogonally along $\phi^*(\Gamma)$. The axial symmetry about $\phi(\Gamma)$ is an isometry of $\mathbb{E}(4H^2-1,H)$, which corresponds to the mirror symmetry about the totally geodesic surface $P$, in which case one can extend analytically both immersions $\phi$ and $\phi^*$ across $\Gamma$ by means of Schwarz reflection principle. In particular, if $\Sigma$ is a surface with piecewise regular boundary such that $\phi$ maps $\partial\Sigma$ onto a geodesic polygon consisting of vertical and horizontal geodesics and the angles at the vertices are integer divisors of $\pi$, then no singularity appears at such vertices after successive reflections about the boundary components, whence $\phi$ (resp.\ $\phi^*$) can be extended by mirror symmetries to a complete minimal immersion (resp.\ $H$-immersion). We refer to~\cite{ManTor,Ple15} for details. Some of these properties are also similar to those of the Lawson duality between minimal surfaces in $\mathbb{R}^3$ and mean curvature $1$ surfaces in $\mathbb{H}^3$, which is not an $\mathbb{E}(\kappa,\tau)$-space (e.g., we refer to the work of Karcher~\cite{K2}).

If $\phi(\Gamma)$ is a horizontal geodesic, let us parametrize $\phi^*(\Gamma)$ with unit speed as $\gamma_*=(\beta,z)\in\mathbb{H}^2\times\mathbb{R}$. Since $\gamma_*$ is contained in a vertical plane orthogonal to $\Sigma$, then $\{\gamma_*',N\}$ is an orthonormal frame along $\gamma_*$ as a curve of that plane, so
\begin{equation}\label{eqn:horizontal-components}|z'|=|\langle\gamma_*',\partial_t\rangle|=\sqrt{1-\langle N,\partial_t\rangle^2}=\sqrt{1-\nu^2}\quad\text{and}\quad \|\beta'\|=\sqrt{1-(z')^2}=|\nu|.
\end{equation}
Therefore, the projections $\beta$ and $z$ might fail to be injective only at points where $\nu$ equals $0$ or $\pm 1$ (this motivates the analysis of the angle function in our construction given by Lemma~\ref{lem:slab}). 

Let us develop further the framework to deal with the case $\phi(\Gamma)$ is a vertical geodesic segment. For convenience, we will identify $\Gamma$ with a unit-speed parametrization $\Gamma:(a,b)\to\Sigma$ such that $\gamma=\phi\circ\Gamma$ satisfies $\gamma'=\xi$, and this orientation will be fixed henceforth. Using the orthonormal frame $\{E_1,E_2,E_3=\xi\}$, we can express
\begin{equation}\label{eqn:N-PQ}
  N_{\gamma(t)}=\cos(\theta(t)) E_1+\sin(\theta(t))E_2,
\end{equation} 
for some function $\theta\in\mathcal{C}^{\infty}(a,b)$ called the \emph{angle of rotation of $N$ along $\gamma$}. As $\phi^*(\Sigma)$ is orthogonal to the slice $\mathbb{H}^2\times\{t_0\}$ in which $\gamma_*=\phi^*\circ\Gamma$ is contained, we infer that the unit normal $N^*$ becomes a unit conormal to $\gamma_*$ in $\mathbb{H}^2\times\{t_0\}$. Since $[E_1,\xi]=0$, we get that $\overline\nabla_\xi E_1=\overline\nabla_{E_1}\xi=H E_1\times E_3=-H E_2$, and $\overline\nabla_\xi E_2=H E_1$ analogously. Taking derivatives in~\eqref{eqn:N-PQ} in the direction of $\gamma'=\xi$, we get to
\[
  \overline\nabla_{\gamma'}N=
  (H-\theta')(\sin(\theta)E_1-\cos(\theta)E_2)=(H-\theta')N\times\gamma'.
\]
Consequently,
\begin{align*}
  H-\theta'&=\langle
              \overline\nabla_{\gamma'}N,N\times\gamma'\rangle=
              -\langle A\alpha',J\alpha'\rangle=
              \langle JA^*\alpha',J\alpha'\rangle-H\langle J\alpha',J\alpha'\rangle\\
            &=\langle
              A^*\alpha',\alpha'\rangle-H=
              -\langle\overline\nabla_{\gamma_*'}N^*,\gamma_*'\rangle-H=\kappa_g-H,
\end{align*} 
where $\kappa_g$ stands for the geodesic curvature of $\gamma_*$ as a curve of $\mathbb{H}^2\times\{t_0\}$ with respect to the conormal $N^*$. This leads to the well-known formula (see~\cite{Ple13})
\begin{equation}\label{eqn:theta}
  \theta'=2H-\kappa_g.
\end{equation}
The direction of rotation of $N$ (i.e., the sign of~$\theta'$) is geometrically meaningful, being the essence of the later distinction between $(H,k)$-noids and $(H,k)$-nodoids.

We will obtain extra properties by assuming that the immersions are \emph{multigraphs}, i.e., their common angle function has no interior zeros. This means that the submersion $\pi$ restricted to the interior of the surface is a local diffeomorphism onto its image, which is a (possibly non-embedded) domain of the base surface.

\begin{lemma}\label{lemma:rotation-curvature}
Let $\phi:\Sigma\to\mathbb{E}(4H^2-1,H)$ and $\phi^*:\Sigma\to\mathbb{H} ^2\times\mathbb{R}$ be conjugate multigraphs over domains $\Omega\subset\mathbb{M}^2(4H^2-1)$ and $\Omega^*\subset\mathbb{H}^2$, respectively, with $0<H\leq\frac{1}{2}$. Assume that $\Gamma$ is a curve in $\partial\Sigma$ such that $\gamma=\phi\circ\Gamma$ satisfies $\gamma'=\xi$.
\begin{enumerate}[label=\emph{(\alph*)}]
  \item If the angle of rotation $\theta$ of $N$ along $\gamma$ is strictly increasing (resp.\ decreasing), then $N^*$ points to the interior (resp.\ exterior) of $\Omega^*$ along $\gamma_*=\phi^*\circ\Gamma$.

  \item If $\theta'>0$ and $\int_{\Gamma}\theta'\leq\pi$, then $\gamma_*$ is embedded.
\end{enumerate}
\end{lemma}

\begin{proof}
The condition $\theta'>0$ (resp.\ $\theta'<0$) implies that $\kappa_g<2H$ (resp.\ $\kappa_g>2H$) in view of Equation~\eqref{eqn:theta}. Therefore, if $N^*$ points outside $\Omega^*$ (resp.\ inside $\Omega^*$) along $\gamma_*$, then the vertical $H$-cylinder sharing the same normal as $\gamma_*$ at any point of $\gamma_*$ lies locally outside $\Omega^*\times\mathbb{R}\subset\mathbb{H}^2\times\mathbb{R}$, which produces a contradiction to the boundary maximum principle for $H$-surfaces.

As for item (b), suppose that $\theta'>0$ and $\int_{\Gamma} \theta'\leq\pi$ but the curve $\gamma_*$ is not embedded. Starting at a non-self-intersection point of $\gamma_*$ and extending the curve on both sides until reaching the first self-intersection point, we obtain a subarc of $\gamma_*$ that encloses a topological disk $D\subset\mathbb{H}^2\times\{t_0\}$ with a vertex (the point of self-intersection) defining an interior angle $0<\alpha\leq 2\pi$. Since $\theta'>0$, we infer that $\kappa_g<2H$ from~\eqref{eqn:theta} and that $N^*$ points towards $\Omega^*$ along $\gamma_*$ by item (a), so $N^*$ points to the exterior of $D$, and the geodesic curvature of $\partial D$ is given by $-\kappa_g$. Gauss-Bonnet formula and~\eqref{eqn:theta} yield
\[
-\text{Area}(D)= \pi+\alpha+ \int_{\partial D}\kappa_g > \pi + 2H\, \text{Length}(\partial D) - \int_{\partial D}\theta' > \pi-\int_{\Gamma}\theta'.
\]
This contradicts the fact that the right-hand side is non-negative by hypothesis.
\end{proof}

The orientation is essential in Lemma~\ref{lemma:rotation-curvature} since we are assuming that $H>0$. In the case $H=0$ the conditions $\theta'>0$ and $\theta'<0$ are equivalent and the embeddedness of $\gamma_*$ in item (b) of Lemma~\ref{lemma:rotation-curvature} also holds true when $\theta'<0$ and $\int_{\Gamma}|\theta'|\leq\pi$.

\begin{remark}\label{rmk:complete-vertical}
If $\gamma$ is a complete vertical geodesic and $\Sigma$ is a graph, it follows that $\theta'$ tends to zero as one approaches the endpoints of $\gamma$. This means that $\gamma_*$ is asymptotic to curves of constant curvature $2H$ on both endpoints, whence it is a divergent curve on the horizontal slice in which it is contained.
\end{remark}

\begin{proposition}[continuity of the conjugation]\label{prop:continuity}
Let $\Sigma$ be a smooth surface and assume that $\phi_n:(\Sigma,\mathrm{d} s^2_n)\to\mathbb{E}(4H^2-1,H)$ is a sequence of minimal isometric immersions that converge on the $\mathcal{C}^m$-topology on compact subsets (for all $m\geq 0$) to a minimal isometric immersion $\phi_\infty:(\Sigma,\mathrm{d} s_\infty^2)\to\mathbb{E}(4H^2-1,H)$. Then the conjugate $H$-immersions $\phi_n^*$ (up to suitable isometries of $\mathbb{H}^2\times\mathbb{R}$) converge in the same mode to the conjugate $H$-immersion $\phi_\infty^*$.
\end{proposition}

\begin{proof}
The fundamental data $(\nu_n,T_n,A_n)$ and the induced metric $\mathrm{d} s_n^2$ are defined in terms of the derivatives of $\phi_n$, and thus converge uniformly on compact subsets of $\Sigma$ to the fundamental data $(\nu_\infty,T_\infty,A_\infty)$ and the metric of $\phi_\infty$. To show that the conjugate surfaces converge, we  first need to adapt them using ambient isometries.

If $\nu_\infty$ is constant $\pm1$, then $\phi_\infty$ is a horizontal slice (in particular $H=0$) and the functions $\nu_n$ converge uniformly to $\pm1$ on compact subsets, which implies that $\phi_\infty^*$ is also a horizontal slice, and the statement follows straightforwardly. Hence, we will assume there is some $x\in\Sigma$ such that $\nu_\infty(x)\neq\pm 1$ and consider a fixed positively oriented basis $\{e_1,e_2\}$ of the tangent plane $T_x\Sigma$ such that $e_1$ (resp.\ $e_2$) is collinear to $T_\infty^*$ (resp.\ $JT^*_\infty$) at $x$. (Note that $T^*_\infty$ and $JT^*_\infty$ do not vanish at $x$ because $\nu_\infty(x)\neq\pm 1$.) Define $q_n=\phi_n^*(x)$, $u_n=(\mathrm{d}\phi_n^*)_x(e_1)$ and $v_n=(\mathrm{d}\phi_n^*)_x(e_2)$ for $n\in\mathbb{N}\cup\{\infty\}$. Since $(T_n)_x\to (T_\infty)_x$ as   $n\to\infty$, we can find isometries $R_n$ of $\mathbb{H}^2\times\mathbb{R}$ preserving both the orientation and  the orientation of the vertical fibers such that $R_n(q_n)=q_\infty$, $(\mathrm{d} R_n)_{q_n}(u_n)\to u_\infty$ and $(\mathrm{d} R_n)_{q_n}(v_n)\to v_\infty$ as $n\to\infty$, i.e., each $R_n$ is a composition of a translation (taking $q_n$ to $q_\infty$) and a   suitable rotation about the vertical axis containing $q_\infty$. Since $\{e_1,e_2\}$ is positively oriented and each $R_n$ is orientation-preserving, also the unit normals satisfy $(\mathrm{d} R_n)_x(N_n^*)\to N_\infty$ as $n\to\infty$. Daniel correspondence is defined up to isometries, so we can substitute $\phi_n^*$ with $R_n\circ\phi_n^*$ and assume henceforth that
\begin{equation}\label{prop:continuity:eqn1}
  \phi_n^*(x)=\phi_\infty^*(x),\qquad
  \lim_{n\to\infty}(\mathrm{d}\phi^*_n)_x= (\mathrm{d}\phi^*_\infty)_x\quad\mbox{and}\quad
  \lim_{n\to\infty}(N^*_n)_x= (N^*_\infty)_x.
\end{equation}

The fundamental data $(\nu_n,JT_n,JA_n+H\,\mathrm{id})$ of $\phi_n^*$ given by~\eqref{eqn:daniel} also converge uniformly on compact subsets of $\Sigma$ to $(\nu_\infty,JT_\infty,JA_\infty+H\,\mathrm{id})$, the fundamental data of~$\phi_\infty^*$. Moreover, given a precompact open subset $U\subset\Sigma$, the uniform convergence of~$A_n^*$ on $U$ implies that the immersions $\phi_n^*$ have uniformly bounded second fundamental form on $U$. If $U$ contains $x$, then all the $\phi_n^*$ have $p_\infty$ as accumulation point, so standard convergence arguments yield the existence of a subsequence of $\phi_n^*$ converging in the $\mathcal{C}^m$-topology for all $m\geq 0$ to some $H$-immersion $\phi^*:U\to\mathbb{H}^2\times\mathbb{R}$ (observe that $\mathbb{H}^2\times\mathbb{R}$ has bounded geometry). Since $\phi^*$ has the same fundamental data as $\phi_\infty^*$, they differ in an ambient isometry because of the uniqueness of Daniel correspondence. In view of the relations given by~\eqref{prop:continuity:eqn1}, we have that $\phi^*(x)=\phi^*_\infty(x)$ and $(\mathrm{d}\phi^*)_x=(\mathrm{d}\phi^*_\infty)_x$, plus $\phi^*$ and $\phi^*_\infty$ have the same unit normal at $x$, so the aforesaid ambient isometry must be the identity, whence $\phi^*=\phi_\infty^*$ on $U$. By means of analytic continuation, we get a subsequence converging to $\phi_\infty^*$ on the whole $\Sigma$.

The above argument implies that any subsequence of the original sequence $\phi_n^*$ has a further subsequence that converges to $\phi^*$ on $\Sigma$. This easily leads to the convergence of $\phi_n^*$ itself.
\end{proof}

\subsection{Minimal graphs}\label{sec:minimal-graphs}

A (vertical) graph in $\mathbb{E}(\kappa,\tau)$ is a section of the submersion $\pi:\mathbb{E}(\kappa,\tau)\to\mathbb{M}^2(\kappa)$ defined over some domain $\Omega\subset\mathbb{M}^2(\kappa)$. The graph is said entire when $\Omega=\mathbb{M}^2(\kappa)$. The mean curvature of the graph of a function $u\in\mathcal{C}^2(\Omega)$ over a zero section $F_0:\Omega\to\mathbb{E}(\kappa,\tau)$ can be expressed in divergence form as 
\[2H=\mathrm{div}\frac{Gu}{\sqrt{1+\|Gu\|^2}},\]
where the divergence and norm are computed with respect to the metric of $\mathbb{M}^2(\kappa)$ and $Gu$ is a vector field in $\mathbb{M}^2(\kappa)$ called the generalized gradient, see~\cite{ManN}. In the cylinder model given by~\eqref{eqn:model-Ekt}, we take as zero section $F_0(x,y)=(x,y,0)$, so the graph is parametrized in terms of $u$ as
\begin{equation}\label{eqn:graph-parametrization}
F_u(x,y)=(x,y,u(x,y)),\qquad (x,y)\in\Omega
\end{equation}
and hence $Gu=(u_x+\tau y\lambda^{-1})\partial_x+(u_y-\tau
x\lambda^{-1})\partial_y$.

The following two entire minimal graphs will play the role of barriers in the construction considered in the next section: 
\begin{itemize}
  \item Given $p\in\mathbb{E}(\kappa,\tau)$, the \emph{umbrella} $\mathcal{U}_p$ centered at $p$ is a complete minimal surface consisting of all horizontal geodesics through $p$. If $\kappa\leq 0$, then $\mathcal{U}_p$ is an entire rotationally invariant minimal graph whose only point with angle function equal to $1$ is $p$. In the cylinder model, the umbrella $\mathcal{U}_{0}$ centered at the origin is given by the graph of the   zero function $u(x,y)=0$.

  \item Given a horizontal geodesic $\Gamma\subset\mathbb{E}(\kappa,\tau)$, there is a complete minimal surface $\mathcal{I}$ that consists of all horizontal geodesics orthogonal to $\Gamma$. It is invariant under isometric translations along $\Gamma$ and we will call it \emph{invariant surface}. If $\kappa\leq 0$, then $\mathcal{I}$ is an entire minimal graph whose angle function equals $1$ only along $\Gamma$. In the cylinder model, if we take $\Gamma$ as the $x$-axis, then $\mathcal{I}$ is the graph of 
  \[u(x,y)=\begin{cases}
  \tau xy&\text{if }\kappa=0,\\
  \frac{2\tau}{\kappa}\arctan\frac{2xy}{\frac{4}{\kappa}+x^2-y^2}&\text{if }\kappa<0.
  \end{cases}\]
\end{itemize}

Now we are going to describe the Jenkins-Serrin minimal graphs used in our construction. Let $\Omega\subset\mathbb{M}^2(\kappa)$ be a convex polygon whose sides are geodesic arcs $A_1,B_1,\ldots,A_k,B_k$ with vertices at some points of $\mathbb{M}^2(\kappa)$ or in the ideal boundary $\partial_\infty\mathbb{H}^2(\kappa)$ if $\kappa<0$, such that no two of the $A_i$ and no two of the $B_i$ have a common endpoint. We mark each $A_i$ with the value $+\infty$ and each $B_i$ with the value $-\infty$. We will consider the Jenkins-Serrin problem of finding a minimal graph $\Sigma\subset\mathbb{E}(\kappa,\tau)$ over $\Omega$, in some particular symmetric cases, with these prescribed boundary values. The boundary of~$\Sigma$ consists of the vertical geodesics projecting onto the interior vertices of $\Omega$, if any; and the asymptotic boundary $\partial_\infty\Sigma$ of~$\Sigma$ consists of the \emph{ideal vertical geodesics} projecting onto the ideal vertices of $\Omega$, if any (in the case $\kappa<0$), and the \emph{ideal horizontal geodesics} projecting onto the arcs $A_i$ and $B_i$.

The solution~$\Sigma$ is obtained as a limit of graphs defined on bounded polygonal domains with constant prescribed boundary data. The infinite boundary values are truncated by some increasing bounded constants and the domain $\Omega$, in the case it is unbounded, is obtained as a limit of an increasing sequence of bounded geodesic
polygons. Hence, Proposition~\ref{prop:continuity} allows us to describe the conjugate $H$-surface of~$\Sigma$:

\begin{corollary}[Conjugation of Jenkins-Serrin graphs]\label{coro:JS}
Let $\Sigma\subset\mathbb{E}(4H^2-1,H)$ be the solution to the Jenkins-Serrin problem described above. The conjugate $H$-surface $\Sigma^*\subset\mathbb{H}^2\times\mathbb{R}$ is a (possibly non-embedded) multigraph.
\begin{enumerate}[label=\emph{(\alph*)}]
  \item Ideal vertical geodesics in $\partial_\infty\Sigma$, if any, become ideal horizontal curves in $\partial_\infty\Sigma^*$ of constant curvature $\pm 2H$ at height $\pm\infty$.

  \item Ideal horizontal geodesics in $\partial_\infty\Sigma$ become ideal vertical geodesics of $\partial_\infty\Sigma^*$.
\end{enumerate}
\end{corollary}

We remark that if $\Sigma$ is a complete graph, then $\Sigma^*$ is a
complete multigraph, and this implies that $\Sigma^*$ is also a graph,
see~\cite[Theorem~1]{ManR}.

\section{On the solution of the Jenkins-Serrin problem}
\label{sec:JS}
Assume that $\kappa\leq 0$ and consider the cylinder model for $\mathbb{E}(\kappa,\tau)$. Given $a,b\in(0,\infty)$ and an integer $k\geq 2$, we define $T_{a,b}\subset\mathbb{M}^2(\kappa)$ as a geodesic triangle of vertices $p_0=(0,0)$, $p_1$ and $p_2$ labeled counterclockwise such that the geodesic segments $\overline{p_0p_1}$ and $\overline{p_0p_2}$ have lengths $a$ and $b$, respectively, and the angle at $p_0$ is equal to $\frac{\pi}{k}$. Up to an isometry fixing the orientation, we can assume that $p_1$ lies on the $x$-axis. The law of cosines gives the length $\ell$ of the third side $\overline{p_1p_2}$:
\begin{equation}\label{eqn:ell}
  \begin{aligned}
    \cosh(\ell\delta)&
    =\cosh(a\delta)\cosh(b\delta)\!-\!\sinh(a\delta)\sinh(b\delta)\cos(\tfrac\pi
    k)& \text{if }\kappa<0,\\
    \ell^2&=a^2+b^2-2ab\cos(\tfrac{\pi}{k})&\text{if }\kappa=0,
  \end{aligned}
\end{equation}
where $\delta=\sqrt{-\kappa}$.

We will also consider the limit cases $T_{a,\infty}=\overline{\cup_{b>0}T_{a,b}}$ and $T_{\infty,b}=\overline{\cup_{a>0}T_{a,b}}$, which are triangles with one ideal vertex if $\kappa<0$ or truncated strips if $\kappa=0$. Note that the point $p_1$ (resp.\ $p_2$) and the segments $\overline{p_ip_1}$ (resp.\ $\overline{p_ip_2}$) should be understood as an ideal point and a geodesic with infinite length in the boundary of $T_{a,b}$ if $a=\infty$ (resp.\ $b=\infty$). After successive axial symmetries about $\overline{p_0p_1}$ and $\overline{p_0p_2}$, we get a polygonal domain $\Omega_{a,b}\subset\mathbb{M}^2(\kappa)$ consisting of $2k$ copies of $T_{a,b}$ whose vertices will be labeled as $p_1,p_2,\ldots,p_{2k}$ counterclockwise. We can lift these points by means of the zero section as $q_i=F_0(p_i)=(p_i,0)\in\mathbb{E}(\kappa,\tau)$ for all $i\in\{0,\ldots,2k\}$.

We are interested in the Jenkins-Serrin problem in $\mathbb{E}(\kappa,\tau)$ over $\Omega_{a,b}\subset\mathbb{M}^2(\kappa)$ with prescribed boundary values $+\infty$ on $\overline{p_{2i-1}p_{2i}}$ and $-\infty$ on $\overline{p_{2i}p_{2i+1}}$, for any $i$. The desired solution containing $q_0$, denoted by $\Sigma_{a,b}$, will be given by Lemmas~\ref{lemma:JS-psl} and~\ref{lemma:JS-nil} below.

\begin{remark}\label{rmk:minimal-case}
Note that $\Sigma_{a,b}$ and $\Sigma_{b,a}$ are not congruent if $a\neq b$ due to the lack of orientation-reversing isometries in $\mathbb{E}(\kappa,\tau)$ if $\tau\neq 0$, see~\cite[Corollary~2.11]{Man},  though both surfaces project onto congruent domains. This implies that the surfaces $\Sigma_{a,b}$ will not be congruent for different values of $a,b$. Indeed, the surface $\Sigma_{b,a}$ is recovered by substituting the value $+\infty$ with $-\infty$ in our Jenkins-Serrin problem over $T_{a,b}$. On the other hand, if $\tau=0$, the surfaces $\Sigma_{a,b}$ and $\Sigma_{b,a}$ (constructed likewise, see~\cite{MR,P}) are actually congruent by means of a mirror symmetry. This is the grounds for many dissimilarities between our $H$-surfaces and the minimal ones in~\cite{MR,P}, as we shall discuss later on.
\end{remark}

\subsection{The case of $\widetilde{\mathrm{SL}}_2(\mathbb{R})$}

We begin by adapting results in the literature to solve this problem for $\kappa<0$, also including the limit problem over $T_{\infty,\infty}=\overline{\cup_{a,b>0}T_{a,b}}$.

\begin{lemma}\label{lemma:JS-psl}  
Let $\kappa<0$ and $\tau\neq 0$. Given $k\geq 2$ and $a,b\in(0,\infty]$, the Jenkins-Serrin problem in $\mathbb{E}(\kappa,\tau)$ over $\Omega_{a,b}\subset\mathbb{H}^2(\kappa)$ with boundary values $+\infty$ on $\overline{p_{2i-1}p_{2i}}$ and $-\infty$ on $\overline{p_{2i}p_{2i+1}}$ for all $i$, has a unique solution $\Sigma_{a,b}$ up to vertical translations.
\end{lemma}

\begin{proof}
If $a,b\in(0,\infty)$, then the existence and uniqueness of solution follows from the Jenkins-Serrin theorem proved by Younes in~\cite{Y}. Likewise, if $a=b=\infty$, then the existence of solution to the problem was given by Melo~\cite{Me}.

Let us now discuss the existence of solution in the case where only one of the parameters $a$ or $b$ is infinite, i.e., only one of the points $p_1$ or $p_2$ is ideal. Assume without loss of generality that $a=\infty$, so that $p_1$ is ideal. Let $\gamma\subset\mathbb{H}^2(\kappa)$ be the complete geodesic containing $\overline{p_1p_2}$. Take a sequence of Jenkins-Serrin graphs over finite triangles with vertices $p_0$, $p_{1,n}$ and $p_2$, being $p_{1,n}\in\gamma$ points diverging to~$p_1$, with prescribed boundary values $0$ on $\overline{p_0 p_{1,n}}\cup \overline{p_0 p_2}$ and $+\infty$ on $\overline{p_{1,n}p_{2}}$. The maximum principle for minimal graphs on bounded domains (proved in~\cite{Y}) implies that this is an increasing sequence of graphs, and we just need to find an appropriate upper barrier that shows that the sequence is nowhere divergent. To this end, consider a geodesic triangle $T'\subset\mathbb{H}^2(\kappa)$ with three ideal vertices and one side equal to $\gamma$ such that $T_{a,b}\subset T'$. By~\cite{Me} we know there exists a minimal graph $\Sigma'$ over $T'$ with boundary values $+\infty$ on $\gamma$ and $0$ on the other two sides of $T'$. The surface $\Sigma'$ lies above any graph in the sequence by the maximum principle in~\cite{Y}. Then the sequence of graphs converge, up to passing to a subsequence, to a minimal graph $\Sigma_\infty$ taking the prescribed boundary values because of the upper barrier $\Sigma'$ and the fact that, given any compact set in $T_{a,b}$, we can take a term in the sequence (in fact infinitely many) used as a barrier from below. In order to obtain the desired solution over $\Omega_{a,b}$, we extend the graph over $T_{a,b}$ by successive symmetries about $\overline{q_0q_1}$ and $\overline{q_0q_2}$ in $\mathbb{E}(\kappa,\tau)$, where we recall that $q_i=F_0(p_i)=(p_i,0)$ for any $i$.

Uniqueness of solutions in the unbounded cases can be proved by following similar arguments as in~\cite{CR,LR}.
\end{proof}

It follows that this Jenkins-Serrin problem over $\Omega_{a,b}$ is equivalent to the Jenkins-Serrin problem over $T_{a,b}\subset\mathbb{M}^2(\kappa)$ with boundary values $0$ on $\overline{p_0p_1}\cup\overline{p_0p_2}$ and $+\infty$ on $\overline{p_1p_2}$ by uniqueness, since a solution of the latter produces all solutions of the former, up to vertical translations, after successive axial symmetries about $\overline{q_0q_1}$ and $\overline{q_0q_2}$ in $\mathbb{E}(\kappa,\tau)$. Observe that the value $0$
implies that $\overline{q_iq_{k+i}}=F_0(\overline{p_ip_{k+i}})$ is a horizontal geodesic for all $i\in\{1,\ldots,k\}$.

\subsection{The case of $\mathrm{Nil}_3$}

Now we deal with the case $\kappa=0$ and $\tau\neq 0$. Up to scaling the ambient metric, we will assume that $\tau=\frac{1}{2}$ and write $\mathrm{Nil}_3=\mathbb{E}(0,\frac{1}{2})$.

First, let us prove that we can find rather explicit solutions to the Jenkins-Serrin problem over $T_{a,b}$ when $k=2$ and $a=\infty$ or $b=\infty$, because they are foliated by (non-necessarily geodesic) straight lines that intersect a given horizontal geodesic segment $\Gamma\subset\mathrm{Nil}_3$ orthogonally. Up to an ambient isometry, we will suppose $\Gamma$ is contained in the $y$-axis and parametrize such a surface as
\begin{equation}\label{eqn:helicoids:X}
X:I\times\mathbb{R}\to\mathrm{Nil}_3,\qquad X(u,v)=\left(u,v,u\,h(v)\right),
\end{equation}
for some smooth real function $h:I\subseteq\mathbb{R}\to\mathbb{R}$ representing the slope of the straight line going through $(0,v,0)$. We shall write $h(v)=\tfrac{1}{2}(v-f(v))$, which simplifies the subsequent calculations (note that translations along $\Gamma$ in the $\mathrm{Nil}_3$-geometry correspond to adding a constant to $f$). We will assume that $0\in I$ and $h(0)=0$, which implies that $X$ is axially symmetric with respect to the $x$-axis (i.e., $f$ and $h$ are odd functions). The mean curvature of the parametrization~\eqref{eqn:helicoids:X} is given by
\[H(u,v)=\frac{u ((4+f(v)^2) f''(v)-2 f(v)(f'(v)-1)(f'(v)-2))}{2 (u^2f'(v)^2-4 u^2f'(v)+f(v)^2+4 u^2+4)^{3/2}}.\]
Hence $X$ is minimal if and only if $f$ satisfies the second-order \textsc{ode}
\begin{equation}\label{eqn:helicoids:minimal}
  (4+f(v)^2) f''(v)=2 f(v) (f'(v)-1)(f'(v)-2).
\end{equation}
Equation~\eqref{eqn:helicoids:minimal} admits the first integral 
\begin{equation}\label{eqn:helicoids:first}
f'(v)=\frac{2\sqrt{4+c^2f(v)^2}}{\sqrt{4+c^2f(v)^2}+c\sqrt{4+f(v)^2}},
\end{equation}
which depends on a constant of integration $c\neq -1$. Given $\mu\in\mathbb{R}$, the solution of~\eqref{eqn:helicoids:minimal} with initial conditions $f(0)=0$ and $f'(0)=1-2\mu$ (i.e., $h'(0)=\mu$) defined on a maximal interval $(-t_\mu,t_\mu)\subseteq\mathbb{R}$ is the inverse of the following odd function:
\begin{equation}\label{eqn:helicoids:g}
  g_\mu(x)=
  \frac{1}{2}\int_0^x\left(1+\frac{(1+2\mu)\sqrt{4+y^2}}{(1-2\mu)\sqrt{4+(\frac{1+2\mu}{1-2\mu})^2y^2}}\right)\mathrm{d} y.
\end{equation}
This follows from an elementary integration of~\eqref{eqn:helicoids:first} by setting $c=\frac{1+2\mu}{1-2\mu}$, and covers all initial values of the first derivative  except for $\mu=\frac{1}{2}$. If $\mu\to\frac{1}{2}$, from above or from below, then the integrand of~\eqref{eqn:helicoids:g} converges monotonically to a function which is not integrable in a neighborhood of $0$. However, the case $\mu=\frac{1}{2}$ corresponds to the constant solution $f(v)=0$ (which has no inverse) and leads to the invariant surface $\mathcal{I}$, see Section~\ref{sec:minimal-graphs} and Remark~\ref{rmk:helicoid-nil-cases}. 

\begin{lemma}\label{lemma:helicoids-nil}
There is a continuous $1$-parameter family $\mathcal{H}_\mu$, $\mu\in\mathbb{R}$, of complete properly embedded minimal surfaces in $\mathrm{Nil}_3$ foliated by straight lines orthogonal to a horizontal geodesic $\Gamma$ and containing a horizontal geodesic orthogonal to $\Gamma$.
\begin{itemize}
  \item If $\frac{-1}{2}\leq \mu\leq\frac{1}{2}$, then $\mathcal{H}_\mu$ is an entire minimal graph.

  \item Otherwise, $\mathcal{H}_\mu$ is a horizontal helicoid-type minimal surface of axis $\Gamma$.
\end{itemize}
The subfamilies with $\mu>\frac{1}{2}$ or $\mu<\frac{-1}{2}$ differ in the direction of rotation, and both of them can be independently reparametrized by the distance between two successive vertical geodesics contained in $\mathcal{H}_\mu$, which ranges from $0$ (when $|\mu|$ goes to $+\infty$) to $+\infty$ (when $|\mu|$ converges to $\frac 12$).
\end{lemma}

\begin{proof}
If $\mu=\frac{1}{2}$, we get $f(v)=0$ and $X$ parametrizes globally the invariant entire minimal graph $\mathcal{I}$. If $\frac{-1}{2}\leq \mu<\frac{1}{2}$, then $c=\frac{1+2\mu}{1-2\mu}\geq 0$ and the integrand in~\eqref{eqn:helicoids:g} is greater than or equal to $1$, whence   $g_\mu$ is a diffeomorphism from $\mathbb{R}$ to $\mathbb{R}$, and so is its inverse $f=g_\mu^{-1}$. This means that $X$ defines an entire minimal graph. 

Assume that $\mu>\frac{1}{2}$ (resp.\ $\mu<\frac{-1}{2}$). The integrand in~\eqref{eqn:helicoids:g} is then strictly negative (resp.\ strictly positive). In particular, $g_\mu$ is strictly monotonic, so it is a diffeomorphism from $\mathbb{R}$ onto its image. Note that $c=\frac{1+2\mu}{1-2\mu}\neq-1$ is negative and a simple algebraic manipulation allows us to rewrite Equation~\eqref{eqn:helicoids:g} as
  \[
    g_\mu(x)=\int_0^x\frac{2(1-c^2)\,\mathrm{d}
      y}{4+c^2y^2+\sqrt{(4+c^2y^2)(4c^2+c^2y^2)}}.
  \]
Therefore, the limit $t_\mu=\lim_{x\to\infty}|g_\mu(x)|$ is a non-zero real number such that $f=g_\mu^{-1}:(-t_\mu,t_\mu)\to\mathbb{R}$ is a diffeomorphism. Since $h(v)$ diverges as $v$ approaches $\pm t_\mu$, the straight lines foliating the surface rotate monotonically having the vertical geodesics $s\mapsto(0,\pm t_\mu,s)$ as limits, see Figure~\ref{Fig-PoligonoH12}. The helicoid-type surface $\mathcal{H}_\mu$ is obtained after successive reflections about these vertical geodesics. However, the above argument yields a different monotonicity of $g_\mu$ (and hence of $h$) for $\mu>\frac{1}{2}$ and $\mu<\frac{-1}{2}$, which reflects the different directions of rotation.

As for the last sentence of the statement, we have shown that two successive vertical geodesics in $\mathcal{H}_\mu$ are $s\mapsto(0,-t_\mu,s)$ and $s\mapsto(0,t_\mu,s)$, whose distance is $2t_\mu$ (any other two successive vertical geodesic give rise to the same distance by symmetry). Then, we will finish by showing that $\mu\mapsto t_\mu$ induces an increasing bijection from $(-\infty,\frac{-1}{2})$ to $(0,+\infty)$ and a decreasing bijection from $(\frac{1}{2},+\infty)$ to $(0,+\infty)$. The  integrand in~\eqref{eqn:helicoids:g}, as a function of $\mu\in(-\infty,\frac{-1}{2})\cup (\frac{1}{2},+\infty)$, satisfies
  \[
    \frac{\partial}{\partial \mu}
    \left(1+\frac{(1+2\mu)\sqrt{4+y^2}}{(1-2\mu)\sqrt{4+(\frac{1+2\mu}{1-2\mu})^2y^2}}\right)=
    \frac{8\sqrt{4+y^2}}{(1-2\mu)^2(4+(\frac{1+2\mu}{1-2\mu})^2y^2)^{3/2}}
    >0.
  \]
Taking into account the absolute value in the definition of $t_\mu$, it follows that $\mu\mapsto t_\mu$ is strictly increasing (resp.\ decreasing) when restricted to $(-\infty,\frac{-1}{2})$ (resp.\ $(\frac{1}{2},+\infty)$). Taking limits in~\eqref{eqn:helicoids:g} by means of the monotone convergence theorem, we get that $\lim_{\mu\to\pm\infty}t_\mu=0$ and $\lim_{\mu\to\pm1/2}t_\mu=+\infty$, so we are done.
\end{proof}

\begin{figure}
  \begin{center}
    \includegraphics[width=\textwidth]{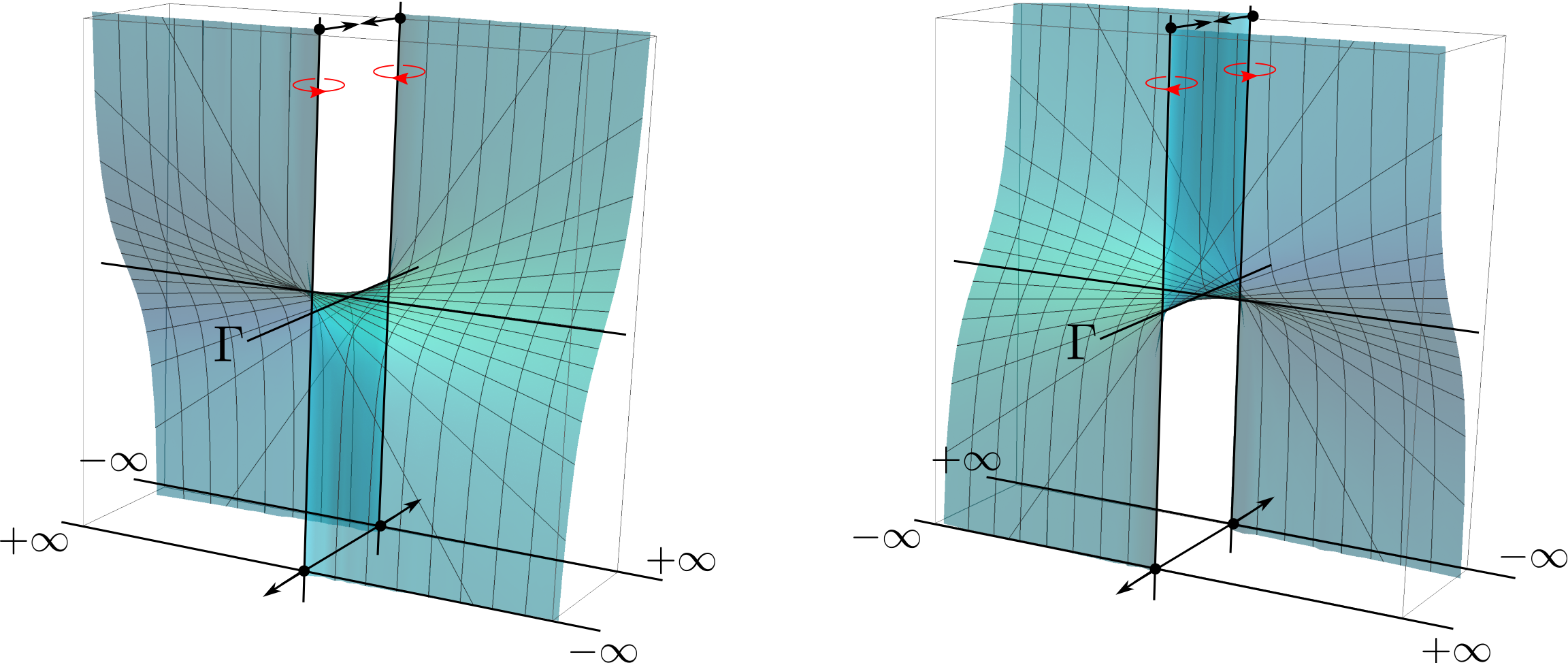}
  \end{center}
  \caption{Rotation of the normal of the horizontal helicoids $\mathcal{H}_\mu$ in the cases $\mu<\frac{-1}{2}$ (left) and $\mu>\frac{1}{2}$ (right).}
  \label{Fig-PoligonoH12}
\end{figure}

\begin{remark}\label{rmk:helicoid-nil-cases}
If $\mu=\pm\frac{1}{2}$, we get $h(v)=\pm\frac{1}{2}v$, and both cases are congruent to the invariant surface $\mathcal{I}$ discussed in Section~\ref{sec:minimal-graphs}. If $\mu=0$, then $h(v)=0$, so $\mathcal{H}_0$ is the umbrella $\mathcal{U}_0$ centered at the origin. Daniel and Hauswirth~\cite[Section~7]{DH} gave a family of   horizontal helicoids which correspond to our case $\mu<\frac{-1}{2}$ in view of~\cite[Lemma~5.1]{Ple13}. Daniel~\cite[Examples~8.4 and~8.5]{Dan11} also constructed two entire minimal graphs foliated by straight lines meeting but they do not contain horizontal geodesics orthogonal to the axis. This reveals that the assumption $h(0)=0$ excludes some cases among all solutions of~\eqref{eqn:helicoids:minimal}. Nevertheless, we have included   that assumption in our study as it will be used in the construction below.
\end{remark}

\begin{remark}\label{rmk:helicoid-nil-congruency}
No two of the surfaces $\mathcal{H}_\mu$ are congruent in $\mathrm{Nil}_3$, except for $\mathcal{H}_{1/2}$ and $\mathcal{H}_{-1/2}$. Among the surfaces with $\mu\in(-\infty,\frac{-1}{2})\cup(\frac{1}{2},+\infty)$, this assertion follows from the fact that isometries of $\mathrm{Nil}_3$ preserve vertical geodesics, distances and orientation. Also, such helicoid-type surfaces are not congruent to any $\mathcal{H}_\mu$ with $\mu\in[\frac{-1}{2},\frac{1}{2}]$ since they are not entire graphs. As for the case $\mu\in(\frac{-1}{2},\frac{1}{2})$, the angle function of $\mathcal{H}_\mu$ in the parametrization~\eqref{eqn:helicoids:X} is given by
\begin{equation}\label{eqn:Hmu-angle}
  \nu(u,v)=\frac{2}{\sqrt{u^2(1-2h'(v))^2+(2h(v)+v)^2+4}},
\end{equation}
whence the origin is the only point of $\mathcal{H}_\mu$ with angle $1$. Since isometries preserve the angle function, the origin must be fixed by a potential isometry. It is not difficult to deduce from~\eqref{eqn:Hmu-angle} and that it must also fix the horizontal  geodesic $\Gamma$, and hence the direction of rotation along it. As the Gauss curvature of $\mathcal{H}_\mu$ at the origin is $\frac{-3}{4}-\mu^2$ and the direction of rotation is different for $\mu>0$ and $\mu<0$, we conclude that no two of the $\mathcal{H}_\mu$ are congruent if $\mu\in(\frac{-1}{2},\frac{1}{2})$.
\end{remark}

We will use these surfaces $\mathcal{H}_\mu$ to prove the
existence of solution of the Jenkins-Serrin problem on $\Omega_{a,b}$ described above.

\begin{lemma}\label{lemma:JS-nil}
Assume that $\kappa=0$ and $\tau\neq 0$. Given $k\geq 2$ and $a,b\in(0,\infty]$ not both of them equal to $\infty$, the Jenkins-Serrin problem in $\mathbb{E}(0,\tau)$ over $T_{a,b}\subset\mathbb{R}^2$ with boundary values $0$ on $\overline{p_0p_1}\cup\overline{p_0p_2}$ and $+\infty$ on $\overline{p_1p_2}$ has a solution.
\end{lemma}

\begin{proof}
We will assume that $\tau=\frac12$ after an homothety of the metric. If $a,b<\infty$, then $\Omega_{a,b}\subset\mathbb{R}^2$ is a bounded triangle, and Pinheiro~\cite{Pinheiro} proved that there exists a unique solution to the problem under this assumption.

In the infinite case, we will begin with $k=2$. If $a=\infty$, there is $\mu>\frac{1}{2}$ such that $t_\mu=b$ by Lemma~\ref{lemma:helicoids-nil}, and part of $\mathcal{H}_\mu$ is a graph over the interior of the half-strip $T_{\infty,b}=[0,+\infty)\times[0,b]$ that solves the desired problem. Likewise, if $b=\infty$, we can find $\mu'<\frac{-1}{2}$ such that $t_\mu=a$, so part of $\mathcal{H}_{\mu'}$ can be written as the desired graph over the interior of $T_{a,\infty}=[0,a]\times[0,+\infty)$ after a rotation of angle $\frac{\pi}{2}$.

Finally, we will assume that $k\geq 3$ and $a=\infty$ (the case $b=\infty$ follows similarly). For each $n\in\mathbb{N}$, let $\Sigma_{n,b}$ be the solution to the corresponding Jenkins-Serrin problem over $T_{n,b}$, whose existence follows from~\cite{Pinheiro}. We observe that $\Sigma_{n,b}$ is a decreasing sequence of minimal graphs, by the maximum principle for graphs over bounded domains proven in~\cite{Pinheiro}. Let $\Sigma'$ be the graph over the strip $[b\cos\frac\pi k,+\infty)\times[0,b \sin\frac\pi k]$ taking boundary values $+\infty$ over $[b\cos\frac\pi k,+\infty)\times\{b \sin\frac\pi k\}$ and $0$ in the remaining part of the boundary. The surface $\Sigma'$ is nothing but part of a translated helicoid $\mathcal{H}_\mu$. By the maximum principle in~\cite{Pinheiro}, all $\Sigma_{n,b}$ are bounded from below by $\Sigma'$. This implies the convergence on compact subsets of a subsequence of $\Sigma_{n,b}$ to a minimal graph over $T_{\infty,b}$, taking the desired boundary values.
\end{proof}

By successive reflections about the horizontal boundary of the surface obtained in Lemma~\ref{lemma:JS-nil}, we obtain the desired solution $\Sigma_{a,b}$ to the Jenkins-Serrin problem on $\Omega_{a,b}$.

\begin{remark}
In the case $a=b=\infty$ (not considered in Lemma~\ref{lemma:JS-nil}), the segment $\overline{p_1p_2}$ disappears and the solution to the Jenkins-Serrin problem over the wedge $T_{\infty,\infty}=\cup_{a,b>0}T_{a,b}$ is not unique, see~\cite{NST}. Nonetheless, we can define $\Sigma_{\infty,\infty}$ as the limit of $\Sigma_{a,b}$ as $a,b\to\infty$. Cartier~\cite[Corollary~3.8]{Cartier} obtained entire graphs in $\mathrm{Nil}_3$ with zero values on $\partial T_{\infty,\infty} $ and asymptotically positive on the interior of $T_{\infty,\infty}$. By using them as barriers from below along with the umbrella $\mathcal{U}_0$, if follows that $\Sigma_{\infty,\infty}\subset\mathrm{Nil}_3$ is an entire minimal graph whose restriction to the interior of $T_{\infty,\infty}$ is strictly positive. If $k=2$, then $\Sigma_{\infty,\infty}$ is the invariant surface $\mathcal{I}$ because of the continuity of the family $\mathcal{H}_\mu$ (see Lemma~\ref{lemma:helicoids-nil}).\end{remark}

The surface $\mathcal{H}_\mu$ can be also parametrized globally as  
\[(s,v)\mapsto (s\cos\alpha(v),v,s\sin\alpha(v)),\] 
where $\alpha(v)=\arctan(h(v))$ for $v\in(-t_\mu,t_\mu)$ and extends by symmetry to all $v\in\mathbb{R}$. This parametrization allows us to compute a unit normal and its rotation along the vertical geodesic
$s\mapsto(0,t_\mu,s)$ as
\begin{equation}\label{eqn:Hmu-theta}
  N=\frac{-1}{\sqrt{1+\sigma^2s^2}}E_1-\frac{\sigma s}{\sqrt{1+\sigma^2s^2}}E_2,\quad \theta'(s)=\frac{\mathrm{d}}{\mathrm{d} s}\arccos\langle N,E_1\rangle=\frac{-\sigma}{1+\sigma^2s^2},
\end{equation}
where 
\[
  \sigma=\alpha'(t_\mu)=\lim_{v\to
    t_\mu}\frac{h'(v)}{1+h(v)^2}=\frac{(1+2\mu)^2}{4\mu}.
\]
This last limit follows from writing $h(v)=\frac{1}{2}(v-f(v))$ and substituting $f'(v)$ by means of~\eqref{eqn:helicoids:first} and using that $\lim_{v\to t_\mu}f(v)=\pm\infty$ when taking the limit.

\section{The conjugate construction}\label{sec:constructions}

Throughout this section, we fix $0<H\leq\frac 1 2$ and $k\geq 2$. We will work in the cylinder model for $\mathbb{E}(4H^2-1,H)$ given by~\eqref{eqn:model-Ekt}. Consider the minimal graph $\Sigma_{a,b}\subset\mathbb{E}(4H^2-1,H)$ over the region $\Omega_{a,b}\subset\mathbb{M}^2(4H^2-1)$ with alternating $\pm\infty$ boundary values constructed in Section~\ref{sec:JS}, being $a,b\in(0,\infty]$. The boundary components of $\Sigma_{a,b}$ are the $2k$ complete vertical geodesics $\Gamma_{i}=\pi^{-1}(p_{i})$ for $i\in\{1,\ldots,2k\}$, plus $2k$ ideal horizontal geodesics at $\mathbb{M}^2(4H^2-1)\times\{\pm\infty\}$. The corresponding curves $\Gamma_i$ must be understood as ideal vertical geodesics if $a=\infty$ or $b=\infty$, and they will be used just heuristically if $H=\frac{1}{2}$ because the ideal boundary of $\mathbb{R}^2$ is not well defined. Our goal now is to describe the conjugate $H$-surface $\Sigma^*_{a,b}\subset\mathbb{H}^2\times\mathbb{R}$.

Although there are some dissimilarities between the cases we will treat, let us first explain some common features that will curtail the forthcoming arguments. Since the angle function is preserved under conjugation, we know that $\Sigma^*_{a,b}$ is a multigraph over a (possibly non-embedded) domain $\Omega^*\subset\mathbb{H}^2$. Furthermore, $\Sigma_{a,b}$ is invariant by reflection about the horizontal geodesics $\overline{q_iq_{k+i}}$, $i\in\{1,\ldots,k\}$, so $\Sigma^*_{a,b}$ has mirror symmetry with respect to $k$ vertical planes meeting at a common vertical line, say the $z$-axis, arranged symmetrically. Corollary~\ref{coro:JS} shows that the boundary components of $\Sigma_{a,b}^*$ are the $2k$ complete (possibly ideal) horizontal curves $\Gamma_1^*,\dots,\Gamma_{2k}^*$ along with $2k$ ideal vertical geodesics joining the endpoints of $\Gamma_i^*$ and $\Gamma_{i+1}^*$ for $i\in\{1,\ldots,2k\}$, see also Remark~\ref{rmk:complete-vertical}.

\begin{enumerate}
  \item If $\Gamma_i$ is not an ideal geodesic, let $\theta_i$ be the angle of rotation along $\Gamma_i$ defined by~\eqref{eqn:N-PQ}. It satisfies $\theta_i'<0$ (resp.\ $\theta_i'>0$) if $i$ is odd (resp.\ even), see  Figure~\ref{Fig-Orientation-st-positive}. Recall that the unit normal $N$ of $\Sigma_{a,b}$ is chosen so that the angle function $\nu=\langle N,\xi \rangle$ is positive. Equation~\eqref{eqn:theta} implies that $\kappa_g>2H$ (resp.\ $\kappa_g<2H$), being $\kappa_g$ the geodesic curvature of $\Gamma^*_i$ as a curve in a horizontal slice with respect to the unit normal $N^*$ of $\Sigma_{a,b}^*$. Lemma~\ref{lemma:rotation-curvature} shows that $N^*$ points to the exterior (resp.\ interior) of $\Omega^*\subset\mathbb{H}^2$ if $i$ is odd (resp.\ even).
  
  \item If $\Gamma_i$ is ideal, then we can reason likewise for a sequence of graphs over finite triangles $T_{a_n,b_n}$, with $a_n,b_n<\infty$, converging to $T_{a,b}$. In that limit, $\theta_i'$ converges uniformly to zero, and  Equation~\eqref{eqn:theta} tells us that $\kappa_g=2H$ with respect to $N^*$. As a limit of curves in the assumption of item (1), we infer that $N^*$ points to the exterior (resp.\ interior) of the domain $\Omega^*$ along $\Gamma_i^*$ if $i$ is odd (resp.\ even). Hence, the geodesic curvature of $\Gamma_i^*$ with respect to $\Omega^*$ is $-2H$ (resp.\ $2H$) if $i$ is odd (resp.\ even).
\end{enumerate}

\begin{figure}
	\begin{center}
		\includegraphics[width=\textwidth]{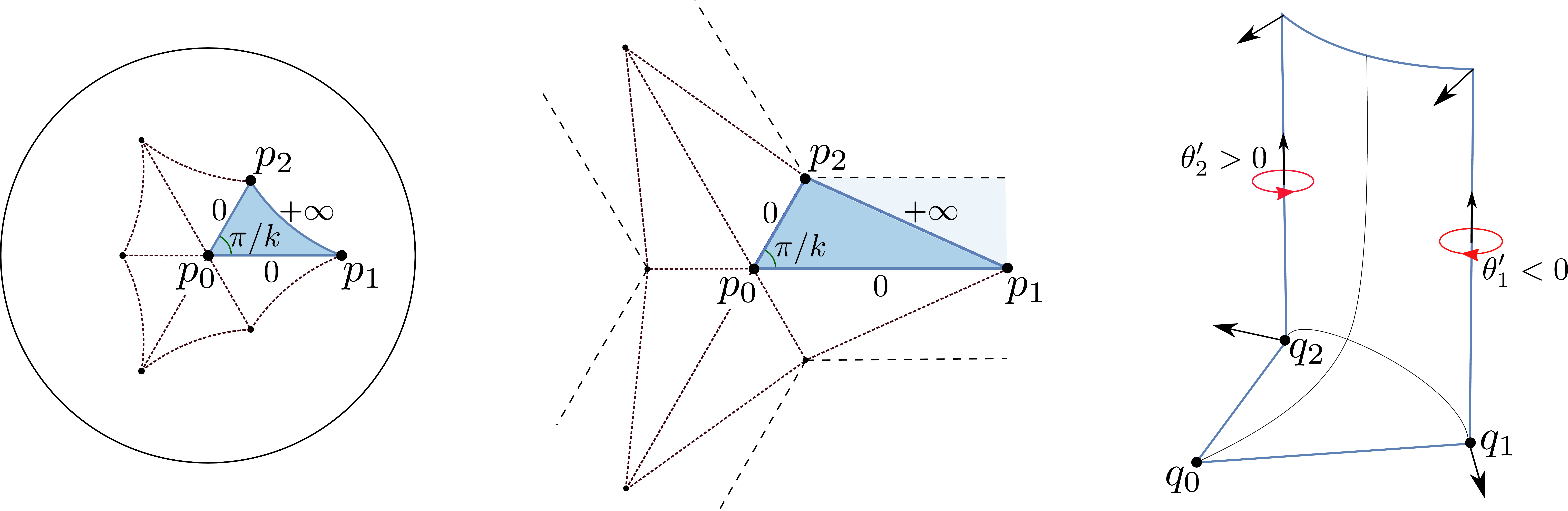}
	\end{center}
	\caption{The domain $\Omega_{a,b}$ with $a,b\in(0,\infty)$ is represented on the left (case $\kappa<0$) and center ($\kappa=0$) for $k=3$. The triangle $T_{a,b}$ is the shaded region and the limit strip $T_{\infty,b}$ appears in lighter color in the central image. On the right, we have sketched $\Sigma_{a,b}$ and the direction of rotation of the normal along $\Gamma_1$ and $\Gamma_2$. We recall that $q_i=F_0(p_i)=(p_i,0)$ for any $i$.} \label{Fig-Orientation-st-positive}
\end{figure}

To see the dependence on the parameters $a$ and $b$, define the function $\rho=\rho(a,b)$ (resp.\ $d=d(a,b)$) as the distance in $\mathbb{H}^2$ from the center $\pi(q_0^*)$ of $\Omega^*$ to the curves $\pi(\Gamma_{2i-1}^*)$ (resp.\ $\pi(\Gamma_{2i}^*)$) in the  projection of the (possibly asymptotic) boundary of $\Sigma^*_{a,b}$. By symmetry, these definitions do not depend on $i\in\{1,\ldots,k\}$. Let us compute the limits $\rho_{\infty}=\rho(\infty,\infty)$ and $d_{\infty}=d(\infty,\infty)$:
\begin{itemize}
  \item If $0<H<\frac{1}{2}$, then $\Sigma_{\infty,\infty}$ is the complete Scherk minimal graph over a symmetric ideal geodesic $2k$-gon, whose conjugate surface is the Scherk $H$-graph in $\mathbb{H}^2\times\mathbb{R}$ over the domain $\Omega^*_{\infty,\infty}$ bounded by $k$ curves of geodesic curvature $2H$ and $k$ curves of geodesic curvature $-2H$ (with respect to $\Omega^*$),   see~\cite[Corollary~1]{ManN}. Therefore, $\rho_\infty$ (resp.\ $d_{\infty}$) is finite and coincides with the distance from the curves of geodesic curvature $-2H$ (resp.\ $2H$) to the center of $\Omega^*_{\infty,\infty}$.
  
  \item If $H=\frac12$, then we have already discussed in
  Section~\ref{sec:JS} that the limit surface $\Sigma_{\infty,\infty}$ is an entire minimal graph, whence $\Sigma_{\infty,\infty}^*$ is an entire $\frac{1}{2}$-graph, see~\cite[Theorem~1.2]{HRS}. This leads to the values $\rho_\infty=d_{\infty}=+\infty$.
\end{itemize}

\begin{lemma}\label{lemma:distance-monotonicity}
The function $d$ is strictly increasing in both variables $a,b\in(0,\infty)$. 
\begin{enumerate}[label=\emph{(\alph*)}]
  \item If $0<H<\frac{1}{2}$, then $d$ is bounded. Furthermore, the functions $a\mapsto d(a,\infty)$ and $b\mapsto d(\infty,b)$ are strictly increasing and range from $0$ to $d_{\infty}$.

  \item If $H=\frac{1}{2}$, then $d(a,\infty)=\rho(\infty,b)=\infty$ for all $a,b\in(0,\infty)$.
\end{enumerate}
\end{lemma}

\begin{proof}
Denote by $\gamma_{a,b}\subset\Sigma_{a,b}$ the horizontal geodesic joining $q_0$ and $q_2$ and by $\nu_{a,b}$ the angle function of $\Sigma_{a,b}$. Since the Jacobian of the projection of $\gamma_{a,b}$ to $\mathbb{H}^2$ is equal to $\nu_{a,b}$, see Equation~\eqref{eqn:horizontal-components}, we infer that $d(a,b)=\int_{\gamma_{a,b}}\nu_{a,b}$. Since $\gamma_{a,b}$ lies in the interior of the graph, $\nu_{a,b}$ does not vanish along $\gamma_{a,b}$, whereas the only point of $\gamma_{a,b}$ where $\nu_{a,b}$ equals $1$ is $q_0$ in view of Lemma~\ref{lem:slab} below.

If $0<a_1<a_2\leq\infty$, then $\Sigma_{a_1,b}$ lies above $\Sigma_{a_2,b}$ as a graph over $T_{a_1,b}$. The boundary maximum principle yields the strict inequality $0<\nu_{a_1,b}<\nu_{a_2,b}<1$ on $\gamma_{a_1,b}-\{q_0\}$. Since $\gamma_{a_1,b}\subseteq\gamma_{a_2,b}$, we deduce that
\[
  d(a_1,b)= \int_{\gamma_{a_1,b}} \nu_{a_1,b}<\int_{\gamma_{a_1,b}}\nu_{a_2,b}\leq\int_{\gamma_{a_2,b}} \nu_{a_2,b}=d(a_2,b).
\]
A similar argument shows that $b\mapsto d(a,b)$ is strictly increasing. If $0<H<\frac{1}{2}$, then using the Scherk minimal graph $\Sigma_{\infty,\infty}$ as a barrier from above, we get that $d(a,b)\leq d_\infty<+\infty$ for all $a,b\in(0,\infty]$, so the strict monotonicity also extends to the cases $a=\infty$ or $b=\infty$. The last assertion in item (a) follows by continuity.

Assume now that $H=\frac{1}{2}$. If $k=2$, then $\Sigma_{a,\infty}$ and $\Sigma_{\infty,b}$ belong to the family $\mathcal{H}_\mu$ with $|\mu|>\frac{1}{2}$. From Equation~\eqref{eqn:Hmu-angle}, we infer that the integral of the angle function of $\mathcal{H}_\mu$ along the $x$-axis is given by
\[
\int_{0}^\infty\frac{\mathrm{d} s}{\sqrt{1+a^2s^2}}=\infty,
\]
for some constant $a$ depending on $\mu$, and thus $d(a,\infty)=\rho(\infty,b)=\infty$. If $k\geq 3$ and $a=\infty$ (the case $b=\infty$ is similar), then $\Sigma_{\infty,b}$ has upper barrier (resp.\ lower barrier) a graph over the strip $[b\cos\frac\pi k,+\infty)\times[0,b \sin\frac\pi k]$ (resp.\ $[0,+\infty)\times[0,b \sin\frac\pi k]$) with boundary values $+\infty$ over $[b\cos\frac\pi k,+\infty)\times\{b \sin\frac\pi k\}$ (resp.\ $[0,+\infty)\times\{b \sin\frac\pi k\}$) and $0$ otherwise. These two barriers are congruent to an element of the family $\mathcal{H}_\mu$, whence the boundary maximum principle implies that the angle function of $\Sigma_{\infty,b}$ is bounded in between two functions whose integral is divergent along the $x$-axis by the above reasoning for the case $k=2$. Therefore, its integral along the $x$-axis also diverges and we conclude that $\rho(\infty,b)=\infty$.
\end{proof}

\begin{remark}
If $H=\frac{1}{2}$, then the equality $\rho(\infty,b)=\infty$ (resp.\ $d(a,\infty)=\infty$) reveals that the curves $\Gamma_{2i-1}^*$ (resp.\ $\Gamma_{2i}^*$) disappear at $\partial_\infty\mathbb{H}^2\times\{\pm\infty\}$ in the limit of $\Sigma_{a,b}^*$ when $a\to\infty$ (resp.\ $b\to\infty$). Note that Corollary~\ref{coro:JS} does not apply because $\Gamma_{2i-1}^*$ (resp.\ $\Gamma_{2i}^*$) is not actually an ideal vertical geodesic. Furthermore, if $\rho(\infty,b)<\infty$ or $d(a,\infty)<\infty$, then such surfaces would contradict Hauswirth, Rosenberg and Spruck's half-space theorem for $\frac{1}{2}$-surfaces~\cite{HRS}.

If $0<H<\frac{1}{2}$, then $d\in(0,d_\infty)$ can be taken as a parameter of the families of $(H,k)$-noids and $(H,k)$-nodoids. This parameter is also valid in the minimal case, being equal to the distance from $\{0\}\times\mathbb{R}$ to the asymptotic geodesic planes and ranging from $0$ to the apothem of an ideal geodesic $2k$-gon. This was proved by Mart\'\i n, Mazzeo and the third author for minimal (horizontal) catenoids in~\cite{MMR}.
\end{remark}

\subsection{The construction of saddle towers}

We will begin by discussing the case $a,b<\infty$ in which none of the curves $\Gamma_1^*,\dots,\Gamma_{2k}^*$ is ideal. The distance between the horizontal planes in $\mathbb{H}^2\times\mathbb{R}$ containing $\Gamma_{i}^*$ and $\Gamma_{i+1}^*$, for any $i$, is the hyperbolic distance $\ell$ between $p_1$ and $p_2$ in $\mathbb{H}^2(4H^2-1)$, given by~\eqref{eqn:ell}. Taking into account that the angle function $\nu=\langle N^*,\xi\rangle$ is positive and $N^*$ points to the exterior (resp.\ interior) of the domain $\Omega^*$ along $\Gamma_i^*$ if $i$ is odd (resp.\ even), we can assume after a vertical translation that $\Gamma_{2i}^*\subset\mathbb{H}^2\times\{0\}$ and $\Gamma_{2i-1}^*\subset\mathbb{H}^2\times\{-\ell\}$ for all $i\in\{1,\ldots,k\}$.

After successive reflections of $\Sigma_{a,b}^*$ about the slices $\mathbb{H}^2\times\{0\}$ and $\mathbb{H}^2\times\{-\ell\}$ we obtain a complete Alexandrov-embedded $H$-surface $\overline\Sigma_{a,b}^*$ invariant by a vertical translation of length $2\ell$, and this proves item (a) of Theorem~\ref{th3}.

\begin{remark}\label{rmk:rescaling}
For each $\ell>0$, Equation~\eqref{eqn:ell} reveals that there is a continuous curve of parameters $(a,b)$ for which the difference of heights of the boundary components of $\Sigma_{a,b}$ is $\ell$. The endpoints of this curve correspond to the parameters $(a,b)=(\ell,0)$ and $(a,b)=(0,\ell)$, not included in the family. It is easy to see that, after a suitable rescaling of the metric keeping $q_0$ fixed, the limit of $\overline\Sigma_{a,b}^*$ as $a\to 0$ for fixed $b>0$ (or $b\to0$ for fixed $a>0$) is a symmetric minimal $k$-noid of $\mathbb{R}^3$, which is non-embedded except for the case $k=2$ that we get a catenoid. That says that, for any $k\geq 3$ there exist non-embedded examples in the family. Note also that rescaling and conjugating commute since rescaling the metric in $\mathbb{E}(\kappa,\tau)$ by a factor $\lambda^2$ results in the transformation of parameters $(\kappa,\tau,H)\mapsto (\lambda^2\kappa,\lambda\tau,\lambda H)$.
\end{remark}

The rest of this section discusses questions related to the embeddedness of saddle towers. By the maximum principle with respect to horizontal slices coming from above, it follows that $\Sigma^*_{a,b}\subset\mathbb{H}^2\times(-\infty,0)$, whereas $\partial\Sigma^*_{a,b}$ lies at heights $0$ and $-\ell$. Hence, $\overline\Sigma_{a,b}^*$ is embedded if and only if $\Sigma^*_{a,b}$ is embedded and lies in the slab $\mathbb{H}^2\times(-\ell,0)$, i.e., above $\mathbb{H}^2\times\{-\ell\}$. On the one hand, $\Sigma_{a,b}^*$ is embedded when its boundary projects one-to-one to $\mathbb{H}^2$. Due to the symmetry of the surface, all boundary components $\Gamma_i^*$ of $\Sigma_{a,b}^*$ are embedded when the conjugate surface $R^*\subset\Sigma_{a,b}^*$ of the fundamental piece $R\subset\Sigma_{a,b}$ projecting onto $T_{a,b}$, is contained in a wedge of angle $\frac{\pi}{k}$. Hence, $\Sigma_{a,b}^*$ is embedded if and only if all the curves $\Gamma_i^*$ are embedded.

Lemma~\ref{lemma:rotation-curvature} yields the
embeddedness of the curves $\Gamma_{2i}^*$ when the interior angle of $\Omega_{a,b}$ at $p_2$ is at most $\pi$. In Section~\ref{sec:k-noids}, we will prove that in the limit case $a=\infty$ the curves $\Gamma_{2i-1}^*$ (at infinity) are embedded. Since the geodesic curvature of $\Gamma_{2i-1}^*$ is bounded and bigger than $2H$ by Equation~\eqref{eqn:theta}, Proposition~\ref{prop:continuity} ensures that $\Gamma_{2i-1}^*$ is embedded when $a$ is large enough.

On the other hand, Lemma~\ref{lem:slab} implies that $\Sigma_{a,b}^*$ lies in the slab $\mathbb{H}^2\times(-\ell,0)$ when $k=2$. It also lies in such slab if $H=0$ (the maximum principle can be also applied to horizontal slices coming from below, see \cite{MR,P}) but this is not true in general for any $H>0$ if $k\geq 3$, as we prove next.

\begin{lemma}\label{lem:slab} 
Fix $H\in(0,\frac{1}{2}]$ and $a,b\in(0,\infty)$. Let $\nu$ be the angle function of $\Sigma_{a,b}\subset\mathbb{E}(4H^2-1,H)$, that will be assumed positive in the interior of $\Sigma_{a,b}$. 
\begin{enumerate}[label=\emph{(\alph*)}]
  \item If $k=2$, then $\nu$ takes the value $1$ only at $q_0$.

  \item If $k\geq 3$, then $\nu$ only takes the value $1$ at $q_0$ and at $k$ points $\hat q_1,\ldots,\hat q_k$ such that $\hat q_i\in\overline{q_0q_{2i-1}}$ for all $i$.
\end{enumerate}
\end{lemma}

\begin{proof}
By symmetry, we will restrict ourselves to the closed fundamental region $R\subset\Sigma_{a,b}$ that projects onto $T_{a,b}$. Observe that $\nu(q_0)=1$ since $\Sigma_{a,b}$ contains at least two horizontal geodesics passing through $q_0$.  Let $q\in R$ be a point such that $\nu(q)=1$, $q\neq q_0$. We will distinguish three cases.

Firstly, suppose that $\pi(q)$ lies in the interior of $T_{a,b}$ and let us reach a contradiction by considering the umbrella $\mathcal{U}_q$ centered at $q$. Since $\mathcal{U}_q$ and $R$ are tangent at $q$, it follows that $\mathcal{U}_q\cap R$ is an equiangular system with at least two curves passing through $q$. As $\mathcal{U}_q$ is contained in a horizontal slab, none of these curves can arrive at a point projecting to the interior of $\overline{p_1p_2}$.  On the other hand, by the maximum principle, $\mathcal{U}_q\cap R$ cannot bound a compact subset in $\mathcal{U}_q$ or in the interior of $R$, because both surfaces are graphs. Moreover, if two curves contained in $\mathcal{U}_q\cap R$ starting at $q$ arrive at $\Gamma_1$ (or $\Gamma_2$), then they would share their endpoint as $\mathcal{U}_q$ is an entire graph, and they would bound a compact subset contained in $\mathcal{U}_q$, a contradiction. Thus, the curves have at most two endpoints in $\Gamma_1\cup\Gamma_2$, and hence (at least) two endpoints $x_1,x_2$ in $\overline{q_0q_1}\cup\overline{q_0q_2}$. The horizontal geodesics (contained in $\mathcal{U}_q$) from $q$ to $x_1$ and $x_2$, along with a subset of $\overline{q_0q_1}\cup\overline{q_0q_2}$, produce a horizontal geodesic triangle or quadrilateral projecting one-to-one to $\mathbb{H}^2(4H^2-1)$, which contradicts the holonomy property of horizontal curves, for the bundle curvature of $\mathbb{E}(4H^2-1,H)$ is not zero.

Secondly, we will assume that $q\in\overline{q_0q_2}$ and reach a contradiction again by using the umbrella $\mathcal{U}_q$. Because of the tangency of $\mathcal{U}_q$ and $\Sigma_{a,b}$, there is at least one curve in $\mathcal{U}_q\cap R$ emanating from $q$ and projecting (locally around $q$) to the interior of $T_{a,b}$. Since there cannot be enclosed compact regions in $\mathcal{U}_q$ or closed horizontal polygons as in the above first case, it is easy to see that such a curve must have an endpoint in $\partial R\cap\Gamma_1=\{(p_1,t):t\geq0\}$. This is a contradiction because the graph $\mathcal{U}_q$ (and hence the horizontal geodesic $\alpha$ joining $q$ and $\Gamma_1$) intersects $\Gamma_1$ at a point whose third coordinate is negative. This is a consequence of the holonomy property because the horizontal polygonal curve $\overline{q_1q_0}\cup\overline{q_0q}\cup\alpha$ projects to a triangle of $\mathbb{H}^2(4H^2-1)$ and produces a positive increment of the third coordinate when we travel on it leaving the interior triangle to our left.

Thirdly and lastly, let us suppose that $q\in\overline{q_0q_1}$. We consider the invariant surface~$\mathcal{I}$ with axis $\overline{q_0q_1}$ (see Section~\ref{sec:minimal-graphs}). If $k=2$, then both $R$ and $\mathcal{I}$ also contain $\overline{q_0q_2}$, and $R$ is above $\mathcal{I}$ on $T_{a,b}$. Since the angle function of $\mathcal{I}$ is constant $1$ along $\overline{q_0q_1}$, the boundary maximum principle ensures that the angle function of $R$ cannot be equal to $1$ on any interior point of $\overline{q_0q_1}$, and this finishes the proof of item (a).

As for the case $k\geq 3$, let $\eta$ be a horizontal unit vector along $\overline{q_0q_1}$ pointing towards the exterior of $T_{a,b}$ and consider the smooth function $\varphi=\langle N,\eta\rangle$ on $\overline{q_0q_1}$, i.e., the cosine of the angle between $N$ and $\eta$. Note that $\varphi(q_0)=0$ and $\lim_{q\to q_1}\varphi(q)=1$.
\begin{itemize}
  \item First, we will prove that $\varphi<0$ on a punctured neighborhood of $q_0$. Therefore, the continuity of $\varphi$ implies that there is an interior point $\hat q$ of $\overline{q_0q_1}$ with $\varphi(\hat q)=0$, which is equivalent to the condition $\nu(\hat q)=1$.

  Observe that $\Sigma_{a,b}$, the invariant surface $\mathcal{I}$ and the umbrella $\mathcal{U}_0$ centered at the origin $q_0$ are three minimal surfaces tangent to each other at $q_0$. On the one hand, $\Sigma_{a,b}\cap \mathcal{U}_0$ contains at least the $k$ horizontal geodesics $\overline{q_i q_{k+i}}$, so all derivatives of $\Sigma_{a,b}$ and $\mathcal{U}_0$ (as graphs) coincide at $p_0$ up to order at least $k-1$. On the other hand, the second derivatives of $\mathcal{I}$ and $\mathcal{U}_0$ do not agree at $p_0$ since $\mathcal{I}\cap\mathcal{U}_0$ consists of just two curves meeting at $q_0$. Since $k\geq 3$, this implies that $\Sigma_{a,b}$ and $\mathcal{I}$ agree at $p_0$ only up to the     first derivatives, so $\Sigma_{a,b}\cap \mathcal{I}$ consists of two orthogonal curves meeting at $q_0$ (in a neighborhood of $q_0$). Since $T_{a,b}$ has angle $\frac\pi k<\frac\pi 2$ at $p_0$, we deduce that $\mathcal{I}$ lies above $\Sigma_{a,b}$ as graphs over the interior of $T_{a,b}$ in a neighborhood of $q_0$. This enables the comparison of $\mathcal{I}$ and $R$ along their common boundary $\overline{q_0q_1}$ near $q_0$ using the     boundary maximum principle. Since $\mathcal{I}$ has constant angle function $1$ along $\overline{q_0q_1}$, we conclude that $\varphi(q)<0$ when $q\in\overline{q_0q_1}$ is close to $q_0$. 

  \item We will now show that the point $\hat q$ is unique. Assume by contradiction that there are two distinct points $\hat q,\hat q'\in\overline{q_0q_1}-\{q_0\}$ with $\nu(\hat q)=\nu(\hat q')=1$. Therefore, there are two curves in $\mathcal{I}\cap R$ projecting to the interior of $T_{a,b}$ and emanating from $\hat q$ and $\hat q'$, respectively. Such curves cannot reach a point projecting to $\overline{p_1p_2}$, as $\mathcal{I}$ is an entire graph, and they do not enclose an interior compact region in $\mathcal{I}$ because of the maximum principle using translated copies of $\mathcal{I}$, so they must reach the vertical fiber $\Gamma_2$. Using that $\mathcal{I}$ is an entire graph once again, we deduce that both curves must have the same endpoint in $\Gamma_2$, so $\mathcal{I}\cap R$ encloses a compact region in $\mathcal{I}$, which is a contradiction.\qedhere
  \end{itemize}
\end{proof}

\begin{remark}\label{rmk:height}
Points of $\Sigma_{a,b}^*$ with $\nu=1$ are the critical points of its height function~$h$ given by the projection to the factor $\mathbb{R}$, which follows from the identity $\|\nabla h\|^2=1-\nu^2$. The conjugate point $q_0^*$ of $q_0$ is a local minimum of $h$ if $k\geq 3$ or a saddle point if $k=2$, see Lemma~\ref{lem:slab}. In the case $k\geq 3$, the conjugate points $\hat q_1^*,\ldots,\hat q_k^*$ of the points $\hat q_1,\ldots,\hat q_k$ given by Lemma~\ref{lem:slab} are saddle points of $h$. This differs with the case $H=0$, where the only critical point of $h$ is $q_0$ for all $k\geq 2$, see~\cite{MR,P}.
\end{remark}

\begin{proposition}\label{prop:slab}
Let $0<H\leq\frac{1}{2}$. 
\begin{enumerate}[label=\emph{(\alph*)}]
  \item If $k=2$, the interior of $\Sigma_{a,b}^*$ lies in the slab $\mathbb{H}^2\times(-\ell,0)$ for all $a,b\in(0,\infty)$.

  \item For each $\ell>0$, there exist (infinitely many) $k\geq 3$ and $a,b\in(0,\infty)$ satisfying~\eqref{eqn:ell} such that $\Sigma_{a,b}^*\not\subset\mathbb{H}^2\times[-\ell,0]$.
\end{enumerate}
\end{proposition}

\begin{proof}
Let $k=2$ and assume by contradiction that the interior of $\Sigma_{a,b}^*$ escapes the open slab $\mathbb{H}^2\times(-\ell,0)$. Since $\partial\Sigma_{a,b}^*\subset\mathbb{H}^2\times\{-\ell,0\}$ and $\Sigma_{a,b}^*\subset\mathbb{H}^2\times(-\infty,0)$, we get there is some interior point $q^*\in\Sigma_{a,b}^*$ where the height function $h$ attains a local minimum. From Lemma~\ref{lem:slab}, the only possibility is $q^*=q_0^*$ and the minimum is global (see Remark~\ref{rmk:height}). Let $\gamma:[0,a]\to\overline{q_0q_1}\subset \Sigma_{a,b}$ be a parametrization with $\gamma(0)=q_0$ and $\gamma(a)=q_1$ and $\gamma^*$ its conjugate curve in $\Sigma_{a,b}^*$. Then $h\circ\gamma^*:[0,a]\to\mathbb{R}$ runs from the global minimum to $-\ell$, and its derivative is negative in a neighborhood of $a$ (as $q_1^*$ lies on $\Gamma_1^*$ where $N^*$ points outside $\Omega^*$). This yields the existence of another critical point of $h\circ\gamma^*$ in $(0,a)$. Since $\Sigma_{a,b}^*$ is orthogonal to a vertical plane along $\gamma^*$ because $\gamma$ is a horizontal geodesic, we infer that the critical point of $h\circ\gamma^*$ is an interior critical point of $h$, a contradiction.

As for item (b), consider a sequence of surfaces $\Sigma_n=\Sigma_{a_n,b_n}$, where $a_n,b_n\in(0,+\infty)$ diverge and satisfy~\eqref{eqn:ell} for $k_n=2^n$ (so that $a,b$ and $k$ appearing in~\eqref{eqn:ell} now depend on $n$ but $\ell$ remains fixed). Note that $q_0\in\Sigma_n$ for all $n$ and $\dist(q_0,\partial\Sigma_n)$ diverges since $k_n$ is increasing whilst $\ell$ is constant. In particular, $\Sigma_n\subset\mathbb{E}(4H^2-1,H)$ is a sequence of minimal graphs with a common point $q_0$. Since graphs are stable, the surfaces $\Sigma_n$ have uniformly bounded second fundamental form, and standard convergence arguments imply the existence of a subsequence (also denoted by $\Sigma_n$) that converges uniformly on compact subsets in the $\mathcal{C}^m$-topology (for all $m$) to a complete minimal surface $\Sigma_\infty\subset\mathbb{E}(4H^2-1,H)$. Since all the $\Sigma_n$ are graphs, then $\Sigma_\infty$ is either a multigraph or a vertical plane, but the latter can be ruled out because the angle functions at $q_0$ equal $1$ along the sequence. As $\Sigma_\infty$ is a complete multigraph, then it was proved in~\cite{ManR} that it must be a graph. For each $n_0\in\mathbb{N}$, there is $n_1\in\mathbb{N}$ such that all the surfaces $\Sigma_n$ for $n\geq n_1$ share the same $2^{n_0}$ horizontal axes of symmetry, all of them with a common point $q_0$. This implies that $\Sigma_\infty$ has also $2^{n_0}$ horizontal axes of symmetry for all $n_0\in\mathbb{N}$, and hence a dense family of such axes, whence $\Sigma_\infty$ is foliated by horizontal geodesics going through $q_0$, i.e., $\Sigma_\infty$ is the umbrella centered at $q_0$. Thus, its conjugate surface $\Sigma_\infty^*\subset\mathbb{H}^2\times\mathbb{R}$ is a rotationally invariant entire $H$-graph, whose height is unbounded. Therefore, there must be some $n_2\in\mathbb{N}$ such that $\Sigma_n^*$ is not contained in the slab $\mathbb{H}^2\times[-\ell,0]$ for any $n\geq n_2$.
\end{proof}

\subsection{The construction of $(H,k)$-noids}\label{sec:k-noids}

We will now treat the case $a=\infty$ and $b\in(0,\infty)$. This means that for any $i\in\{1,\ldots,k\}$, the vertex $p_{2i-1}$ of $\Omega_{\infty,b}$ becomes ideal, and hence $\Gamma_{2i-1}$ is an ideal vertical geodesic (when $0<H<\frac 12$).

Since $\Sigma_{\infty,b}^*$ is the limit of the surfaces
$\Sigma_{a,b}^*$ as $a\to\infty$, we get that all the curves $\Gamma_{2i}^*$ lie in the same horizontal slice, that will be assumed to be $\mathbb{H}^2\times\{0\}$. We can then extend $\Sigma_{\infty,b}^*$ by mirror symmetry about $\mathbb{H}^2\times\{0\}$ to obtain a proper Alexandrov-embedded complete $H$-surface $\overline\Sigma_{\infty,b}^*$. These surfaces will be called \emph{$(H,k)$-noids} (or \emph{$H$-catenoids} if $k=2$).

If $0<H<\frac{1}{2}$, the asymptotic boundary of $\Sigma_{\infty,b}^*$ consists of $2k$ ideal vertical half-lines in $\partial_\infty\mathbb{H}^2\times\mathbb{R}$ together with the $k$ complete curves $\Gamma_{2i-1}^*\subset\mathbb{H}^2\times\{-\infty\}$.  Therefore, $\Gamma_{2i-1}^*$ projects to a curve of constant geodesic curvature $-2H$ in $\mathbb{H}^2$ with respect to $N^*$, that points towards the interior of $\Omega^*$, whence $\overline\Sigma^*_{\infty,b}$ has $k$ ends asymptotic to vertical $H$-cylinders and the surface lies
(locally) in their concave side (the case $k=2$ is shown in
Figure~\ref{Fig-Catenoids}, left).

\begin{figure}
  \begin{center}
    \includegraphics[width=0.95\textwidth]{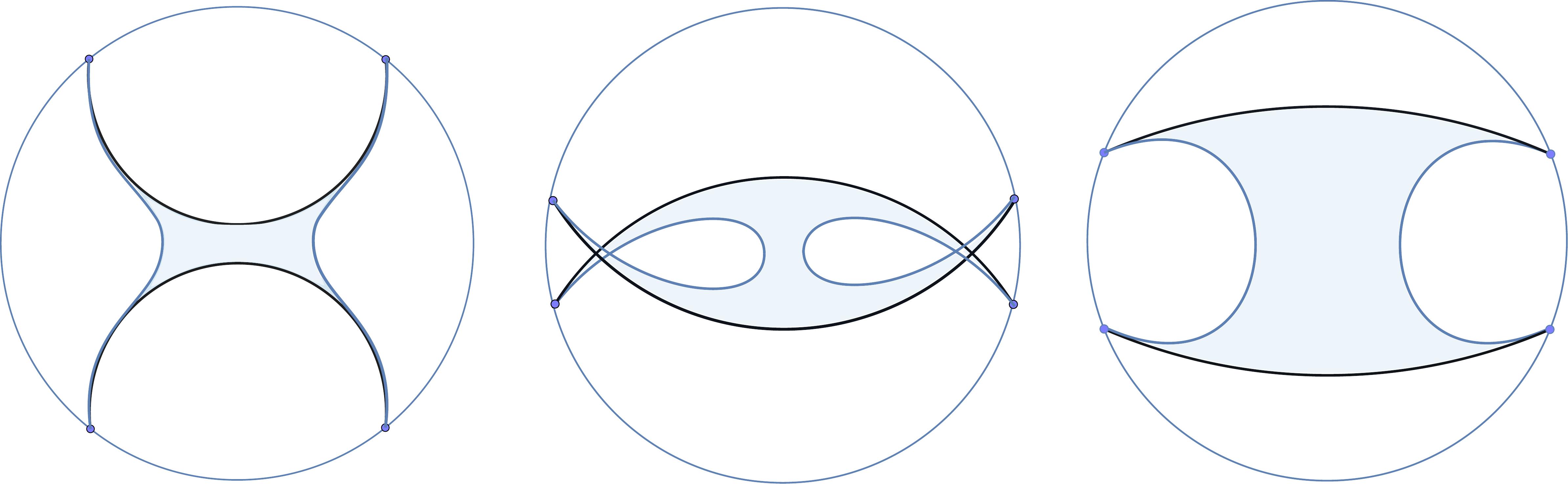}
  \end{center}
  \caption {Sketch of the domain of a $H$-catenoid (left), a
    non-embedded $H$-catenodoid (center), and an embedded
    $H$-catenodoid (right), with $0<H<\frac{1}{2}$.}
\label{Fig-Catenoids}
\end{figure}

The case $H=\frac{1}{2}$ is similar, but the curves $\Gamma_{2i-1}^*$ do not exist by Lemma~\ref{lemma:distance-monotonicity} (we can think that they go to a point as $H$ goes to $\frac 12$). Item (b) of Theorem~\ref{th3} will be proved after analyzing the nature of the ends when $H=\frac12$. If $k=2$, then $\Sigma_{\infty,b}$ belong to the family $\mathcal{H}_\mu$ for $\mu>\frac{1}{2}$ and $\mu<\frac{-1}{2}$, respectively. Equations~\eqref{eqn:theta} and~\eqref{eqn:Hmu-theta} show that $\Gamma_{1}^*$ has geodesic curvature 
\begin{equation}\label{eqn:kg-critical}
  \kappa_g=2H-\theta'_1=1+\frac{4\mu(1+2\mu)^2}{16\mu^2+(1+2\mu)^4s^2}.
\end{equation}

Although Lemma~\ref{lemma:rotation-curvature} already implies the embeddedness of $\frac{1}{2}$-catenoids ($\mu<\frac{-1}{2}$), Equation~\eqref{eqn:kg-critical} allows us to give a precise numerical picture of their domain as $\frac{1}{2}$-bigraphs (see Figure~\ref{fig:critical-catenoid}, left). The case $k\geq 3$ is not explicit, but the asymptotic behavior should be similar because $\Sigma_{\infty,b}$ lies above the graph $\Sigma_-$ over the strip $[b\cos\frac\pi k,+\infty)\times[0,b \sin\frac\pi k]$ with boundary values $+\infty$ over $[b\cos\frac\pi k,+\infty)\times\{b \sin\frac\pi k\}$ and $0$ otherwise, and below a vertical translation $\Sigma_+$ of $\Sigma_-$ by a certain finite distance. The boundary maximum principle implies that the angle of rotation of $\Sigma_{\infty,b}$ along $\Gamma_2$ lies between the angles of rotation of $\Sigma_-$ and $\Sigma_+$ (which differ in a translation of the parameter) and hence they are asymptotically equivalent.

\begin{figure}
\begin{center}
\includegraphics[width=0.67\textwidth]{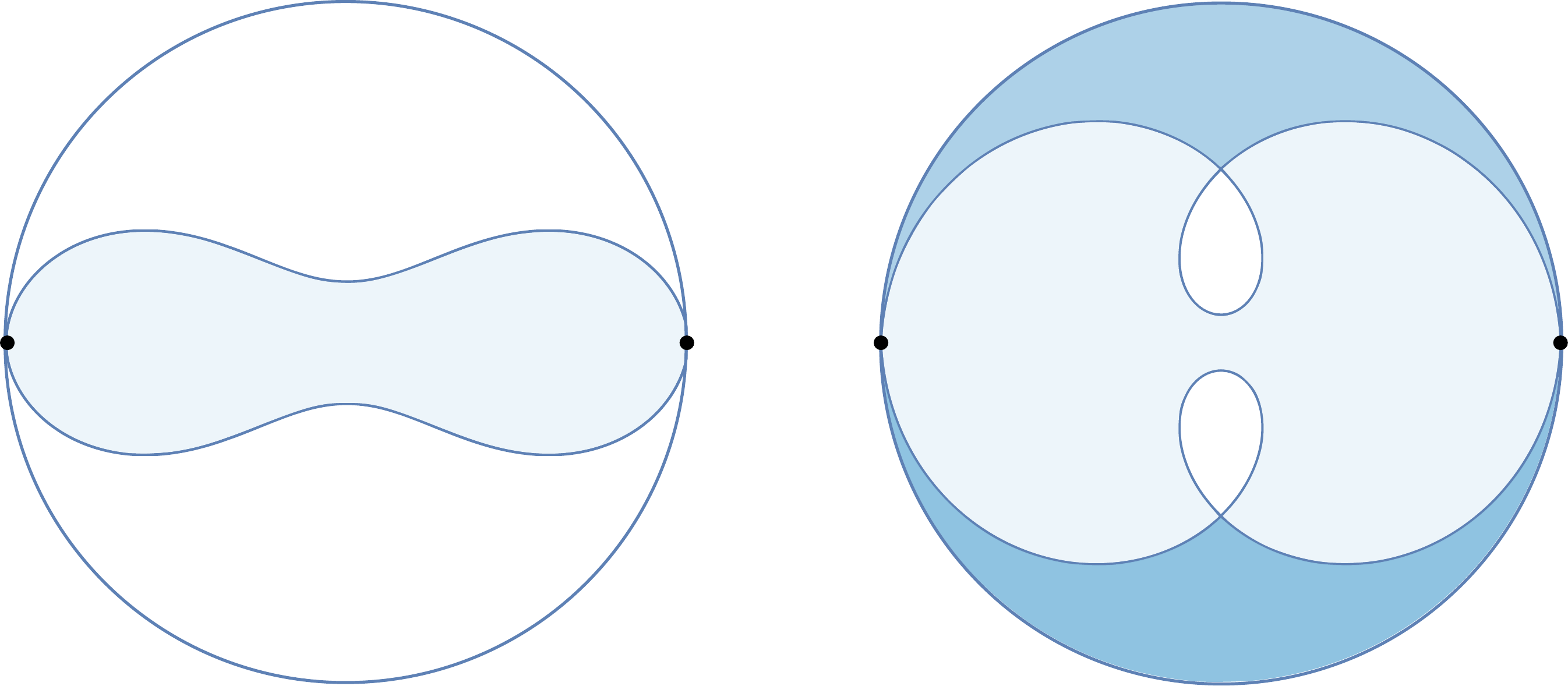}
\end{center}
\caption{Numerical representation of the domains of $\frac{1}{2}$-catenoids (left) and $\frac{1}{2}$-catenodoids (right), corresponding to $\mathcal{H}_\mu$ with $\mu=-3$ and $\mu=3$, respectively. The darker region is covered twice and illustrates why $\frac{1}{2}$-catenodoids are never  embedded.}\label{fig:critical-catenoid}
\end{figure}

As for embeddedness in general (for any $H\in(0,\frac{1}{2}]$ and $k\geq 2$), the maximum principle using horizontal slices yields $\Sigma_{\infty,b}^*\subset\mathbb{H}^2\times(-\infty,0)$. Hence, $\overline\Sigma_{\infty,b}^*$ is embedded if and only if $\Sigma_{\infty,b}^*$ is embedded. Since each curve $\Gamma_{2i-1}^*$ is embedded and $\Sigma_{\infty,b}^*$ is a multigraph, we deduce (again using the symmetry of the surface, as in the above case of saddle towers) that $\overline\Sigma_{\infty,b}^*$ is embedded if and only if the curves $\Gamma_{2i}^*$ are embedded. By Lemma~\ref{lemma:rotation-curvature}, this holds true if the interior angle of $\Omega_{\infty,b}$ at $p_2$ is at most $\pi$. If $0<H<\frac{1}{2}$ and the angle at $p_2$   is $\pi$, basic hyperbolic geometry yields the value $b=\frac{1}{\sqrt{1-4H^2}}\arccosh(\frac{1}{\sin\frac{\pi}{k}})$, which is equal to the distance from $p_0$ to the intersection point $\overline{p_1p_3}\cap\overline{p_0p_2}$. If $H=\frac12$ and $k>2$, then no such value for $b$ exists. The following result summarizes this discussion (observe that this is still not the announced counterexample to the Krust-type property because $\Omega$ is convex under the given assumptions):
\begin{proposition}\label{prop:knoids-embeddedness}~
\begin{enumerate}[label=\emph{(\alph*)}]
  \item If $0<H<\frac{1}{2}$ and
    $b\geq\frac{1}{\sqrt{1-4H^2}}\arccosh(\frac{1}{\sin\frac{\pi}{k}})$,
    then $\overline\Sigma_{\infty,b}^*$ is embedded.
  \item If $0<H\leq\frac{1}{2}$ and $k=2$, then
    $\overline\Sigma_{\infty,b}^*$ is embedded for all
    $b\in(0,\infty)$.
  \end{enumerate}
  \end{proposition}

\subsection{The construction of $(H,k)$-nodoids}\label{sec:k-nodoids}

We will finally deal with the case $b=\infty$ keeping $a<\infty$, which leads to the \emph{$(H,k)$-nodoids} (or \emph{$H$-catenodoids} if $k=2$). Recall that these surfaces are not congruent to $(H,k)$-noids because we are assuming that $H\neq 0$ (see Remark~\ref{rmk:minimal-case}).

Building upon the framework given at the beginning of
Section~\ref{sec:constructions}, the boundary of the minimal graph $\Sigma_{a,\infty}\subset\mathbb{E}(4H^2-1,H)$ consists of the vertical geodesics $\Gamma_{2i-1}$, whereas the ideal vertical geodesics $\Gamma_{2i}$ lie in its asymptotic boundary (if $0<H<\frac 12$). The boundary of the conjugate $H$-multigraph $\Sigma_{a,\infty}^*\subset\mathbb{H}^2\times\mathbb{R}$ is made up of the horizontal symmetry curves $\Gamma_{2i-1}^*$ that can be assumed to be contained in $\mathbb{H}^2\times\{0\}$ up to a vertical translation. As curves of this slice, they satisfy $\kappa_g>2H$ with respect to $N^*$, that points to the exterior of the domain $\Omega^*$
of the multigraph.

If $0<H<\frac{1}{2}$, then the asymptotic boundary of $\Sigma_{a,\infty}^*$ also contains the $k$ ideal curves $\Gamma_{2i}^*\subset\mathbb{H}^2\times\{+\infty\}$ whose projections onto $\mathbb{H}^2$ have constant geodesic curvature $2H$ with respect to the conormal $N^*$, which points inside $\Omega^*$. In other words, the surface $\overline\Sigma_{a,\infty}^*$ obtained from $\Sigma_{a,\infty}^*$ by a mirror symmetry about $\mathbb{H}^2\times\{0\}$ lies locally in the convex part of $\pi^{-1}(\Gamma_{2i}^*)$ in contrast to the case of $(H,k)$-noids (in Figure~\ref{Fig-Catenoids} the case $k=2$ has been depicted). If $H=\frac{1}{2}$, then the asymptotic horocycles $\Gamma_{2i}^*$ disappear at infinity by Lemma~\ref{lemma:distance-monotonicity}, and item (c) of Theorem~\ref{th3} is proved.

In this case, we cannot in general deduce the inclusion $\Sigma_{a,\infty}^*\subset\mathbb{H}^2\times(0,+\infty)$, for the maximum principle does not apply due to the fact that $N^*$ points upwards. However, Lemma~\ref{lem:slab} and the maximum principle do yield the inclusion $\Sigma_{a,\infty}^*\subset\mathbb{H}^2\times(0,+\infty)$ if $k=2$ as a limit of the inclusions $\Sigma_{a,b}^*\subset\mathbb{H}^2\times(0,\ell)$ as $b\to\infty$.

If $0<H<\frac12$, the surface $\Sigma_{a,\infty}^*$ tends to an ideal Scherk $H$-graph as $a\to\infty$ (see Section~\ref{sec:krust}). In this limit, the geodesic curvature of $\Gamma_{2i-1}^*$ converges to $-2H$, so the inclusion $\Sigma_{a,\infty}^*\subset\mathbb{H}^2\times(0,+\infty)$ holds true when $a$ is large enough. On the other hand, Lemma~\ref{lemma:distance-monotonicity} shows that the hyperbolic distance from $\pi(\Gamma_{2}^*)$ to the origin ranges from $0$ to $d_{\infty}$ as $a$ runs from $0$ to $\infty$. Since $\overline\Sigma_{a,\infty}^*$ lies in the convex side of the cylinders over $\Gamma_{2i}^*$, there exists $a_1>0$, depending on $k$ and $H$, such that the asymptotic equidistant curves $\Gamma_2^*$ and $\Gamma_4^*$ in $\Sigma_{a_1,\infty}^*$ share at least one endpoint at infinity. This implies that the endpoints of the curve $\Gamma_3^*$ coincide when $a=a_1$ and $\Gamma_3^*$ has self-intersections if and only if $a<a_1$ (by symmetry, the same happens for any $\Gamma_{2i-1}^*$). Therefore $\Sigma_{a,\infty}^*$ (and hence $\overline\Sigma_{a,\infty}^*$) is not embedded in that cases (see Figure~\ref{Fig-Catenoids}, center).

If $H=\frac12$, then the curves $\Gamma_{2i}^*\subset\overline\Sigma_{a,b}^*$ get close to horocycles at the same time that they diverge to infinity by Lemma~\ref{lemma:distance-monotonicity}.  However, $\overline\Sigma_{a,b}^*$ lies locally in their convex side ($N^*$ points towards the interior of $\Omega^*$ along $\Gamma_{2i}^*$). Therefore, the curves $\Gamma_{2i-1}^*$ cannot be embedded, so neither is $\Sigma_{a,\infty}^*$ (see Figure~\ref{fig:critical-catenoid}). Altogether, we deduce:

\begin{proposition}\label{prop:knodoids-embeddedness}~
\begin{enumerate}[label=\emph{(\alph*)}]
  \item If $0<H<\frac12$, there exist constants $a_1, a_2$ (depending on $k$ and $H$) such that $\overline\Sigma_{a,\infty}^*$ is not embedded for any $a<a_1$ and is embedded for all $a>a_2$.
  
  \item If $H=\frac12$, then $\overline\Sigma_{a,\infty}^*$ is not embedded for any $a>0$.
\end{enumerate}
\end{proposition}

\begin{remark}
  If $k=2$, then $a_1=a_2$ in item (a) by Proposition~\ref{prop:slab}, as the $(H,k)$-nodoids are obtained as limits of the corresponding saddle towers. In the case $k\geq 3$ we know that $\Sigma_{a,\infty}^*$ is embedded if and only if $a\geq a_1$ but in general we cannot assure that it is contained in $\mathbb{H}^2\times(0,+\infty)$.
\end{remark}

\section{The Krust property}\label{sec:krust}
This last section is devoted to prove Theorem~\ref{th1}.

Let us first consider $0<H\leq\frac{1}{2}$. Given $k=2$ and
$a\in(0,\infty)$, the domain $\Omega_{a,\infty}$ is either an ideal quadrilateral in $\mathbb{H}^2(4H^2-1)$ or a strip in $\mathbb{R}^2$, in particular a convex domain. However, Proposition~\ref{prop:knodoids-embeddedness} says that $\Sigma_{a,\infty}^*$ is not embedded if $H=\frac{1}{2}$ or $a<a_1$ and $0<H<\frac12$. Hence $\Sigma_{a,\infty}$ gives the desired counterexample in that case.

We now assume $H>\frac{1}{2}$. Let $T\subset\mathbb{S}^2(4H^2-1)$ be a convex equilateral triangle of vertices $p_0$, $p_1$ and $p_2$ labeled counterclockwise. Consider a geodesic polygon in $\mathbb{E}(4H^2-1,H)$ consisting of three horizontal geodesic segments projecting to the sides of $T$ and two vertical segments $\Gamma_1$ and $\Gamma_2$ projecting to $p_1$ and $p_2$, respectively. Up to vertical translations, this pentagon only depends on the length $a$ of $\Gamma_1$, once $T$ is fixed. Since $\pi^{-1}(T)\subset\mathbb{E}(4H^2-1,H)$ is a mean convex body in the sense of Meeks and Yau, we can find a minimal graph $\Sigma\subset\mathbb{E}(4H^2-1,H)$, which is a graph over the interior of $T$ spanning the pentagon, see~\cite[Proposition~2]{ManTor}. Here, we are assuming that $\mathbb{E}(4H^2-1,H)$ is not globally a Berger sphere but the Riemannian $3$-manifold given by~\eqref{eqn:model-Ekt} in which vertical fibers have infinite length. The angle of rotation of the upward-pointing normal to $\Sigma$ along $\Gamma_1$ (resp.\ $\Gamma_2$) is increasing (resp.\ decreasing), in a similar way to Figure~\ref{Fig-Orientation-st-positive}. The corresponding curve $\Gamma_1^*$ (resp.\ $\Gamma_2^*$) has geodesic curvature larger (resp.\ smaller) than $2H$ in view of Equation~\eqref{eqn:theta}. By making $a$ large enough, we can obtain a region $R\subset\Sigma$ bounded by two almost horizontal curves and two vertical subarcs of $\Gamma_1$ and $\Gamma_2$ of arbitrarily large length and geodesic curvature arbitrarily close to $2H$. This is the shaded region depicted on the left-hand side of Figure~\ref{fig:krust}, foliated by curves very close to horizontal geodesics along which the angle function is almost zero. Their conjugate curves are almost vertical geodesics joining $\Gamma_1^*$ and $\Gamma_2^*$ (which lie at different heights).

\begin{figure}
\begin{center}
\includegraphics[width=0.8\textwidth]{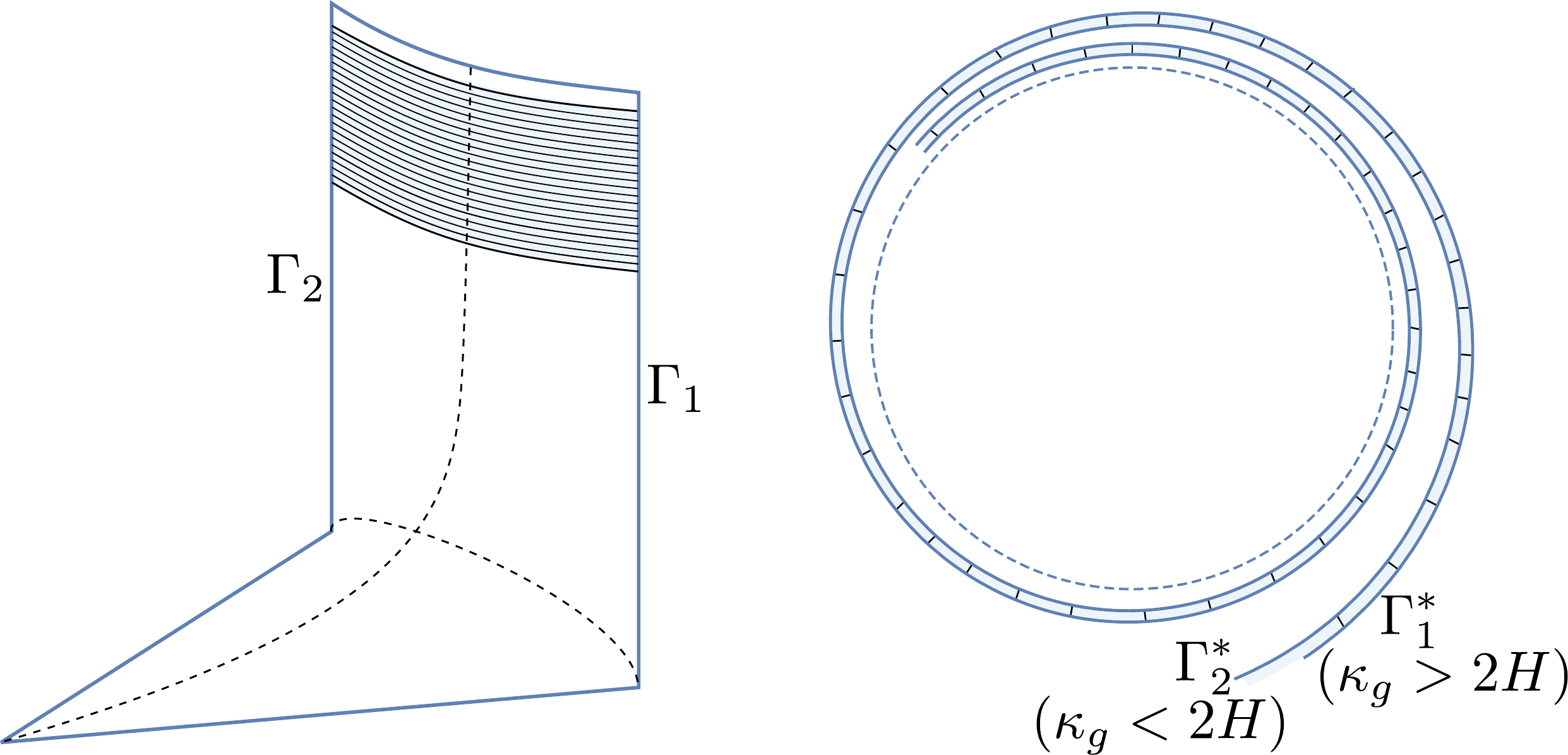}
\end{center}
\caption{The geodesic polygon that produces a counterexample to the Krust property for $H>\frac12$ (left) and the projection of the conjugate of the shaded region (right). The interior circle of $\mathbb{H}^2$ of curvature $2H$ is represented in dashed line.}\label{fig:krust}
\end{figure}

Assume by contradiction that the Krust property holds. Hence the corresponding region $R^*\subset\Sigma^*$ is a graph. Then $\Gamma_1^*$ projects to a curve very close to an $H$-cylinder over a circle of $\mathbb{H}^2$ and its subarc in $\partial R^*$ approaches the circle from outside because its geodesic curvature is larger to $2H$. If $a$ is chosen large enough so that the subarc goes several times around the circle, then the graphical condition implies that $\Gamma_2^*\cap\partial R^*$ is intertwined with the subarc of $\Gamma_1^*$, so their projections spiral together from outside the circle as shown in the right-hand side of Figure~\ref{fig:krust}. This contradicts the fact that $\Gamma_2^*$ has geodesic curvature smaller than $2H$.

\end{document}